\documentclass[a4paper,12pt,reqno,twoside]{amsart}

\usepackage[utf8]{inputenc} 		% UTF8 encoding
\usepackage[T1]{fontenc} 

\usepackage{amssymb}
\usepackage{amsmath}
\usepackage{graphicx}
\usepackage[colorlinks,citecolor=blue,urlcolor=blue]{hyperref}
\usepackage{cite}
\usepackage[hmargin=2.5cm,vmargin=2.5cm]{geometry}
\usepackage{mathrsfs} % Script fonts

\newcommand{\be}{\begin{eqnarray}}
\newcommand{\ee}{\end{eqnarray}}

\newcommand{\beq}{\begin{equation}}
\newcommand{\eeq}{\end{equation}}
\newcommand{\beqn}{\begin{equation*}}
\newcommand{\eeqn}{\end{equation*}}

\newcommand{\slot}{\,\cdot\,}
\newcommand{\round}[1]{\lfloor#1\rfloor}
\newcommand{\fract}[1]{\{#1\}}

\newtheorem{thm}{Theorem}[section]
\newtheorem{prop}[thm]{Proposition}
\newtheorem{cor}[thm]{Corollary}
\newtheorem{lem}[thm]{Lemma}
\newtheorem{defn}[thm]{Definition}

\newtheorem{remark}[thm]{Remark}
\newtheorem{example}[thm]{Example}

\newcommand\cD{{\mathcal D}}

\newcommand\cK{{\mathcal K}}
\newcommand\cL{{\mathcal L}}
\newcommand\cM{{\mathcal M}}

\newcommand\cP{{\mathcal P}}

\newcommand\cT{{\mathcal T}}

\newcommand\bE{{\mathbb E}}

\newcommand\bN{{\mathbb N}}

\newcommand\bP{{\mathbb P}}

\newcommand\bR{{\mathbb R}}
\newcommand\bS{{\mathbb S}}

\newcommand\bZ{{\mathbb Z}}

\newcommand\rd{{\mathrm d}}

\newcommand\fF{{\mathfrak F}}

\newcommand\fS{{\mathfrak S}}

\newcommand\fm{{\mathfrak m}}

\newcommand\fs{{\mathfrak s}}

\newcommand{\ve}{\varepsilon}

\def\bfE{\mathbf{E}}

\def\bfP{\mathbf{P}}

\def\bfT{\mathbf{T}}

\def\bfV{\mathbf{V}}

\begin{document}

\title{Quasistatic dynamical systems}

\author[Neil Dobbs]{Neil Dobbs}
\address[Neil Dobbs]{
Department of Mathematics and Statistics, P.O.\ Box 68, Fin-00014 University of Helsinki, Finland.}
\email{neil.dobbs@gmail.com}
\urladdr{http://www.maths.tcd.ie/~ndobbs/maths}
\author[Mikko Stenlund]{Mikko Stenlund}
\address[Mikko Stenlund]{
Department of Mathematics and Statistics, P.O.\ Box 68, Fin-00014 University of Helsinki, Finland.}
\email{mikko.stenlund@helsinki.fi}
\urladdr{http://www.math.helsinki.fi/mathphys/mikko.html}


%\subjclass[2000]{\#; \#, \#}  % Does not support MSC 2010
 % Suggesting these. 
 
% http://www.ams.org/msc/msc2010.html
% 37D20  	Uniformly hyperbolic systems (expanding, Anosov, Axiom A, etc.)
% 37C40  	Smooth ergodic theory, invariant measures
% 37C60  	Nonautonomous smooth dynamical systems
% 60G44  Martingales with continuous parameter
% 60H10  	Stochastic ordinary differential equations

%\date{\today. {\bf Please do not circulate!}}

\begin{abstract}

We introduce the notion of a quasistatic dynamical system, which generalizes that of an ordinary dynamical system. Quasistatic dynamical systems are inspired by the namesake processes in thermodynamics, which are idealized processes where the observed system transforms (infinitesimally) slowly due to external influence, tracing out a continuous path of thermodynamic equilibria over an (infinitely) long time span. 
Time-evolution of states under a quasistatic dynamical system is entirely deterministic, but choosing the initial state randomly renders the process a stochastic one.
In the prototypical setting where the time-evolution is specified by strongly chaotic maps on the circle, we obtain a description of the statistical behaviour as a stochastic diffusion process, under surprisingly mild conditions on the initial distribution, by solving a well-posed martingale problem. 
We also consider various admissible ways of centering the process, with the curious conclusion that the ``obvious'' centering suggested by the initial distribution sometimes fails to yield the expected diffusion.

\end{abstract}

\maketitle

%%%%%%%%%%%%%%%%%%%%%%%%%%%%
%%%%%%%%%%%%%%%%%%%%%%%%%%%%

\subsection*{Acknowledgements}
This work was supported partially by the Academy of Finland. Mikko Stenlund acknowledges the ERC grant MPOES, Jane and Aatos Erkko Foundation, and Emil Aaltosen S\"a\"ati\"o. The authors are grateful to Carlangelo Liverani for helpful correspondence as well as Dario Gasbarra and Kalle Kyt\"ol\"a for useful remarks.

%%%%%%%%%%%%%%%%%%%%%%%%%%%%%%%%%%%%
%%%%%%%%%%%%%%%%%%%%%%%%%%%%%%%%%%%%
%%%%%%%%%%%%%%%%%%%%%%%%%%%%%%%%%%%%

\medskip
\section{Introduction}
This paper belongs to the wider context of studying statistical properties of dynamical systems using probabilistic techniques. We are going to propose a new, but very natural, class of \emph{non-stationary} dynamical systems, and to implement methods from probability theory (e.g., coupling) and stochastic analysis (a martingale problem) in their investigation, alongside methods from ergodic theory and dynamical systems.

\medskip
\subsection{Motivation and abstract setup.}\label{sec:motivation}

This paper grew out of the will to understand the following abstract setup. Consider a system~$\fs$ in contact with an ambient system~$\fS$. The joint system~$(\fs,\fS)$  is too large and complex to analyze, while~$\fs$, the observed subsystem of actual interest, is more amenable to scrutiny. 
The observed system~$\fs$ and the ambient system~$\fS$ are allowed to interact, with the following constraints: (i)~due to the size of the latter, the evolution of~$\fs$ has a negligible effect on that of~$\fS$, and (ii)~$\fs$ is robust in comparison to the strength of the interaction. 
By~(ii) we mean that the characteristics of~$\fs$ change \emph{very} slowly, yet in such a way that the cumulative change over a \emph{very} long time could be enormous. 

\medskip
At the conceptual level, the problem is to characterize the properties of~$\fs$ over a (long) lapse of time, at least in nontrivial special cases. 
In particular, it is not always reasonable to expect that the subsystem~$\fs$ reaches or remains in equilibrium under the perpetual influence of~$\fS$. Rather, it might slowly pass through small vicinities of equilibria, in a continuous manner. Given an absence of equilibrium states, the question arises as to how to describe statistical properties of such a system.

\medskip
In order to model the above setup, we propose the notion of a {\bf quasistatic dynamical system} in Section~\ref{sec:dtQDS}. In Section~\ref{sec:descQDS} we introduce a means of describing properties of quasistatic dynamical systems via stochastic diffusion processes. 
We follow this with a presentation of continuous-time quasistatic dynamical systems and subsequently a discussion of relevant literature. 
In Section~\ref{sec:model} we present a specific class of quasistatic dynamical systems, whose properties we investigate in the remainder of the paper. Section~\ref{sec:results} contains the statements of theorems, with an overview of the proofs in Section~\ref{sec:outline}.

\medskip
\subsection{(Discrete-time) Quasistatic dynamical systems} \label{sec:dtQDS}
Throughout the paper, we use the standard notation
\beqn
\text{$\fract{u} = u \pmod 1$ \quad and \quad $\round{u} = u-\fract{u}$} 
\eeqn
 for the fractional and integer part of a real number~$u\ge 0$, respectively.

\medskip
\begin{defn}\label{defn:QDS}
Let $X$ be a set and $\cM$ a collection of self-maps $T:X\to X$ equipped with a topology. Consider a triangular array 
\beqn
\bfT = \{T_{n,k}\in\cM\ :\ 0\le k\le n, \ n\ge 1\} 
\eeqn
of elements of~$\cM$. If there exists a piecewise-continuous curve $\gamma:[0,1]\to\cM$ such that
\beqn
\lim_{n\to\infty}T_{n,\round{nt}} = \gamma_t \ , \quad t\in[0,1]\ ,
\eeqn
we say that $(\bfT, \gamma)$ is a \emph{quasistatic dynamical system (QDS)}. 
\end{defn}

\medskip
\begin{defn}\label{defn:state_system_space} We call~$X$ the \emph{state space} (also \emph{phase space}) and~$\cM$ the \emph{system space} of the QDS. 
\end{defn}

\medskip
Obviously a QDS is a generalization of an ordinary dynamical system: a discrete-time dynamical system is determined by a single map $T:X\to X$, which is the degenerate~QDS obtained by setting $T_{n,k}=T$ for all $k$ and $n$. (Then $\gamma_t = T$ for all $t$.) 
Below we will also provide the definition of a continuous-time QDS which similarly generalizes the notion of an ordinary continuous-time dynamical system (semiflow).
One can readily come up with further generalizations of the definitions given, but spelling them out here seems to add little of substance for the purposes of this paper.

\medskip
Let us justify the definition above. An element of the state space $X$ represents the state of the system, whereas the triangular array~$\bfT$ describes the dynamics: given the initial state~$x\in X$ and integers $1\le k\le n$, the image 
\beqn
x_{n,k} = T_{n,k}\circ\dots\circ T_{n,1}(x) \in X
\eeqn
is the state of the system after~$k$ steps on the~$n$th level of the array~$\bfT$. Passing to the continuous time parameter $t\in[0,1]$ via the scaling $k = \round{nt}$, we note that the piecewise-constant curve~$t\mapsto T_{n,\round{nt}}$ approximates~$\gamma$ in the system space~$\cM$, and this approximation converges (pointwise) as $n\to\infty$. The system $T_{n,n}\approx \gamma_1$ differs significantly from the system $T_{n,1}\approx \gamma_0$ in typical situations.

\medskip
The rate of convergence of $T_{n, \round{nt}}$ to $\gamma_t$ will affect the properties of the QDS and may need to be specified when studying a particular system. 
Limiting properties will certainly depend on the limit curve $\gamma$ and may or may not depend on the array $\bfT$, see Section~\ref{sec:results}. 

\medskip
One cannot directly study a ``limit system'', for letting first $n \to \infty$ one would end up just iterating the map $\gamma_0$ and none other. Rather one must study (for example, statistical) properties of the QDS at level $n$ and see how those properties behave in the limit. 
The main goal of this paper is to initiate a systematic study of such systems that would eventually result in a comprehensive mathematical theory of~QDSs.

\medskip
Drawing a parallel with the first paragraph of this section, the limit curve~$\gamma$ models the evolution of the observed system~$\fs$, as the ambient system~$\fS$ forces~$\fs$ to transform, possibly significantly, over time. The jumps in~$\gamma$ allow for singular, abrupt events. For example, such events could include jumps from one connected component of the system space~$\cM$ to another.

\medskip
We have borrowed the apt term ``quasistatic'' from thermodynamics. There it refers to idealized processes in which the observed system transforms infinitesimally slowly due to external influence. Such a system is in thermodynamic equilibrium at any given time, yet traces out a continuous path of different equilibria over an infinitely long time span. For instance, all reversible thermodynamic processes are quasistatic. 
See, e.g.,~\cite{Mandl_1988,Douce_2011} for more background. The connection with thermodynamics should not be construed as a limitation of scope; the definition of QDS is, of course, purely abstract and could be used widely according to one's needs.

\medskip
A natural class of QDS is provided by the following example:
\begin{example}[Quasistatic billiards]
Dispersing billiards on a torus consists of strictly convex scatterers with smooth boundaries embedded in the (surface of the) torus. A particle moves on the torus, in the exterior of the scatterers, and experiences an elastic collision when it meets a scatterer. 
Assuming that the length of free flight between any two successive collisions is uniformly bounded, the location of the particle can be kept track of by keeping track of the collisions. 
For a fixed scatterer configuration~$K$, this leads to a representation of the dynamics by a billiard map $F_{K}$ mapping one collision to the next. A model of dispersing billiards with moving scatterers was introduced in~\cite{StenlundYoungZhang_2013}: Let~$(K_k)_{k=0}^\infty$ be a sequence of scatterer configurations, such that $d(K_{k-1},K_k)<\ve$ holds uniformly for some small~$\ve>0$. Here~$d$ is a natural distance on the space~$\cK$ of admissible configurations.
Then $F_{K_k,K_{k-1}}$ represents the dynamics between the~$(k-1)$th and~$k$th collisions, during which the configuration has changed by a distance~$<\ve$, from~$K_{k-1}$ to~$K_k$, and the compositions $F_{K_k,K_{k-1}}\circ\dots\circ F_{K_1,K_0}$ represent the dynamics for scatterers moving with speed~$<\ve$. 
This yields a quasistatic dynamical system in the limit~$\ve\to 0$ of infinitesimally slowly moving scatterers. More precisely, consider a triangular array of configurations
$
\{K_{n,k}\in\cK\ :\ 0\le k\le n, \ n\ge 1\} 
$
with $d(K_{n,k-1},K_{n,k})<\ve_n$, where $\lim_{n\to\infty}\ve_n = 0$. Assume moreover that there is a continuous curve $\gamma^\cK:[0,1]\to\cK$ such that $\lim_{n\to\infty}K_{n,\round{nt}} = \gamma^\cK_t$. Setting $T_{n,k} = F_{K_{n,k},K_{n,k-1}}$ for~$1\le k\le n$ and~$n\ge 1$ results in a QDS, which we coin quasistatic billiards. (The system space~$\cM$ can be identified with~$\cK\times\cK$, in which case the limit curve~$\gamma$ takes values in the diagonal subspace: $\gamma_t = (\gamma^\cK_t,\gamma^\cK_t)$.)
\end{example}

%\pagebreak
\medskip
\subsection{Statistical description of a QDS} \label{sec:descQDS}
In order to gather information about the state $x_{n,k}$ of the quasistatic dynamical system, one performs a measurement of an observable quantity, whose value is determined by the state. Accordingly, the quantity is represented by a function $f:X\to\bR$, called an ``observable'', and the outcome of the measurement by $f(x_{n,k})$.
We are going to study the sequence of measurements
\beqn
f(x_{n,0}),f(x_{n,1}),\dots,f(x_{n,n})
\eeqn
 on the $n$th level of the QDS, corresponding to the initial state $x_{n,0} \equiv x$, as $n\to\infty$. For practical reasons, such as sensitive dependence on the initial state (aka chaos), it is natural to assume that~$x$ is a random variable with values in $X$ and some distribution~$\mu$. This renders~$(x_{n,k})_{k=0}^n$ a sequence of random variables. 
 By definition, each~$x_{n,k}$ is completely determined by~$x$, and its distribution is the pushforward\footnote{Assuming the map $T$ is measurable, the pushforward of a measure $\mu$ is defined by $T_*\mu(E) = \mu(T^{-1}E)$ for all measurable sets $E$.} $\mu_{n,k} = (T_{n,k}\circ\dots\circ T_{n,1})_*\mu$, which will be different for different~$n$ and~$k$. In other words, the random variables~$x_{n,k}$ are neither independent nor identically distributed.

\medskip
The random paths
\beqn
t\mapsto S_n(x,t) = \int_0^{nt} f(x_{n,\round{s}})\,\rd s \ 
\eeqn
are piecewise-linear interpolations of Birkhoff sums: if $nt \in \bN$, $S_n(x,t) = \sum_{j=0}^{nt-1} f(x_{n,j})$. 
Understanding the distribution of these paths for a large class of observables in the limit~$n\to\infty$ is a natural goal and can be considered a good statistical description of a QDS. 
It is reasonable to hope that the random paths,
once properly centered and scaled, converge to a stochastic diffusion process. We shall prove that this indeed happens for a 
 paradigmatic~QDS whose system space consists of strongly chaotic expanding circle maps, formally presented in Section~\ref{sec:model}. 
The result holds for rather generic~$f$ and~$\mu$; we only ask that $f$ be Lipschitz continuous and $\mu$~absolutely continuous. The curve $\gamma$ need not be especially regular; H\"older continuity suffices. The limit process will depend on~$\gamma$ and~$f$, but not on~$\mu$. See Section~\ref{sec:results} for the precise statements of the main results. 

\medskip
To get an idea why passing to the limit $n\to\infty$ might yield a diffusion process for the paths, consider the following simple example:
\begin{example}[A degenerate case]\label{ex:degenerate}
    Let $T = T_{n,k} : x \mapsto 2x \pmod 1$ for all $n \geq 1$ and $0 \leq k \leq n$, so the curve $\gamma$ is constant, $\gamma_t = T$ for all $t$. 
    Define the observable $f : [0,1) \to \bR$  by $f(x)= -1$ for $x \in [0,\frac12)$ and $f(x) =1$ otherwise. Taking Lebesgue measure for the initial distribution $\mu$, the Birkhoff sums model a fair coin toss, or the simple symmetric random walk. The paths $n^{-1} S_n(x,\slot )$ converge in distribution to the zero path $t \mapsto 0$. The random paths $n^{-\frac12} S_n(x, \slot)$ converge in distribution to standard Brownian motion. 
\end{example}

\medskip
\subsection{Continuous-time QDS} \label{sec:ctQDS}
For completeness, we finish this section with a discussion on continuous-time dynamical systems. 
In continuous time, an ordinary dynamical system is determined by a semiflow, i.e., a one-parameter family of maps $\phi^s : X\to X$, $s\in\bR_+ = [0,\infty)$, satisfying the semigroup property $\phi^s\circ\phi^r = \phi^{s+r}$ with~$\phi^0 = \mathrm{id}_X$. Typical semiflows are those generated by a vector field: if the differential equation
\beqn
\frac{dy}{ds} = V(y) 
\quad\text{with}\quad
y(0) = x 
\eeqn
specified by the vector field~$V$ on~$X$ has a unique solution $y=y(x,s)$ for all $x\in X$, then $\phi^s(x) \equiv y(x,s)$ defines a semiflow. Note that $\frac{d}{ds}\phi^s(x)|_{s=0} = V(x)$.
Semiflows can also depend on time explicitly, as is the case with a time-dependent vector field: if the differential equation
\beqn
\frac{dy}{ds} = V(y,s) 
\quad\text{with}\quad
y(r) = x 
\eeqn
specified by the time-dependent vector field~$V(\slot,s)$ on~$X$ has a unique solution $y=y(x,r,s)$ for all $x\in X$ and all $r\ge 0$, then $\varphi^{r,s}(x) \equiv y(x,r,s)$ defines a time-dependent semiflow; see below. Note that $\frac{d}{ds}\varphi^{r,s}(x)|_{s=0} = V(x,r)$. 
Time-dependent semiflows on~$X$ become (ordinary) semiflows on~$X\times\bR_+$ simply by keeping track of the time-coordinate explicitly. 
More precisely, let~$\mathrm{p}_1$ and~$\mathrm{p}_2$ be the canonical projections from~$X\times\bR_+$ to~$X$ and~$\bR_+$, respectively. 
Then a semiflow $\phi^s:X\times\bR_+\to X\times\bR_+$, $s\in\bR_+$, having the property $\mathrm{p}_2(\phi^s(x,r)) = r+s$ defines the two-parameter family of maps $\varphi^{r,s} : X\to X : \varphi^{r,s}(x) \equiv \mathrm{p}_1(\phi^s(x,r))$, $(r,s)\in\bR_+^2$, having the characteristic property 
\beq\label{eq:nonauto_semiflow}
\varphi^{r,s+u} = \varphi^{r+s,u}\circ\varphi^{r,s}
\quad\text{with}\quad
\varphi^{r,0} = \mathrm{id}_X
\eeq
of time-dependent semiflows. Note that~\eqref{eq:nonauto_semiflow} is consistent with the following interpretation: given the state~$x$ of the system at time~$r$, $\varphi^{r,s}(x)$ is the state $s$ time units later. Conversely, given a time-dependent semiflow $\varphi^{r,s}$ on $X$, 
\beq\label{eq:semiflows}
\phi^s(x,r) \equiv (\varphi^{r,s}(x),r+s)
\eeq
determines a semiflow on $X\times\bR_+$.

\medskip
\begin{defn}\label{defn:continuous_QDS}
Let $X$ be a differentiable manifold and $\cM$ a collection of vector fields on~$X$ equipped with a topology. Let $n$ be a parameter taking values either in $\bZ_+$ or $\bR_+$. Consider the array
\beqn
\bfV  = \{V_n(\slot,r)\in\cM\ : \ 0\le r \le n,\  n\ge 0 \}
\eeqn
and assume that, for each $n\ge 0$, the time-dependent vector field $V_n(\slot,r)$ determines a time-dependent semiflow $\varphi_n^{r,s}$, $0\le r\le r+s\le n$. 
If there exists a piecewise-continuous curve $\gamma:[0,1]\to\cM$ such that
\beqn
\lim_{n \to\infty} V_n(\slot,nt) = \gamma_t\ , \quad t\in[0,1]\ ,
\eeqn
we say that $(\bfV,\gamma)$ is a \emph{continuous-time QDS}.
\end{defn}

\medskip
\begin{defn}
The nomenclature of Definition~\ref{defn:state_system_space} continues to apply in the continuous-time context. We also say that the array $\mathbf{\Phi} = \{\varphi_n^{r,s}\ : \ 0\le r\le r+s\le n,\  n\ge 0 \}$ is a \emph{quasistatic semiflow}.
\end{defn}

\medskip
The convergence condition in the definition means that in order to observe a change of a fixed order of magnitude in the vector field $V_n$ at the $n$th level of the array, one has to wait a time of order $n$. 
In the limit $n\to\infty$ the vector field changes infinitesimally slowly, yet traces out the curve $\gamma$ from beginning to end. For illustrative purposes only, here is a very simple example of a quasistatic semiflow:

\medskip
\begin{example}
Let $X$ be the unit circle in the complex plane, and let $\omega_0,\omega_1$ be two real numbers. Then $\varphi_n^{r,s}(x) = x\exp(i\omega_0 s + \frac i2 (\omega_1-\omega_0) s(s+2r)n^{-1})$ defines a time-dependent semiflow on $X$; see~\eqref{eq:nonauto_semiflow}. In fact, the vector field
\beqn
V_n(x,r) \equiv \frac{d}{ds}\varphi_n^{r,s}(x)\Big|_{s=0} = i\!\left(\omega_0 + (\omega_1-\omega_0) r n^{-1}\right)\! x
\eeqn
describes rotation at angular speed $\omega_0 + (\omega_1-\omega_0) r n^{-1}$. We have 
\beqn
\lim_{n\to\infty}V_n(x,nt) = i\!\left(\omega_0 + (\omega_1-\omega_0) t \right)\! x \equiv \gamma_t(x)\ , \quad t\in[0,1]\ .
\eeqn
Thus, this particular quasistatic semiflow describes a system whose angular speed on the circle changes infinitesimally slowly from~$\omega_0$ to~$\omega_1$.
\end{example}

\medskip
Again, one can readily come up with generalizations of the above definition (e.g.\ regarding the linear time-scaling $r=nt$), but we do not record them here. Let us, however, mention one class of models not quite falling under the definition provided. To compare with the earlier discrete-time case, set
\beqn
\varphi^{r,s}_n = T_{n,\round{r+s}}\circ\dots\circ T_{n,\round{r}+1} \ , \quad 0\le r\le r+s\le n \ .
\eeqn
Then~\eqref{eq:nonauto_semiflow} holds. Note that the convergence condition in Definition~\ref{defn:QDS} now becomes
\beqn
\varphi^{nt-1,1}_n = T_{n,\round{nt}} \to \gamma_t\ ,
\eeqn
which in this context makes more sense than the condition in Definition~\ref{defn:continuous_QDS}.

\medskip
Just as in the discrete-time case, it is interesting to study a quasistatic semiflow in terms of the limit behaviour of the random paths
\beqn
t\mapsto S_n(x,t) = \int_0^{nt} f(x_{n,s})\,\rd s \ ,
\eeqn
where
\beqn
x_{n,s} = \varphi^{0,s}_n(x) \ ,
\eeqn
as $n\to\infty$. A good statistical description of the QDS at issue entails proving limit laws for a large class of observables~$f$ and initial measures~$\mu$.

%%%%%%%%%%%%%%%%%%%%%%%%%%%%%%%

\medskip
\subsection{Preceding literature}
This work is philosophically a natural successor of the sequence of papers \cite{LasotaYorke_1996,OttStenlundYoung_2009,Stenlund_2011,StenlundYoungZhang_2013} on the analysis of time-dependent dynamical systems, although in scope it differs drastically. 
In \cite{LasotaYorke_1996}, ``asymptotic similarity'' of pushforward densities under sequences of (piecewise) expanding maps was established using transfer (Perron--Frobenius) operator techniques. 
Asymptotic similarity is an indication of statistical \emph{memory loss}: the system quickly forgets the distribution of its initial state. 
In~\cite{OttStenlundYoung_2009,Stenlund_2011,StenlundYoungZhang_2013}, a different approach for establishing memory loss for progressively more complicated time-dependent systems was developed based on the idea of \emph{coupling}. 
See~\cite{Young_1999,BressaudLiverani_2002,ChernovDolgopyat_2009,Chernov_2006} for implementations of coupling in deterministic dynamics, and \cite{Lindvall_2002} for an introduction to the probability theory of coupling. 
Let us also mention~\cite{GuptaOttTorok_2013,MohapatraOtt_2014}, which use yet another, Hilbert projective metric, technique for a time-dependent system, as well as~\cite{Romain_etal_2014}. The source of memory loss in the systems of the preceding papers is sensitive dependence on initial conditions (aka\ chaos). 
In the other extreme, memory loss is also produced by sinks; the  references closest to the time-dependent setup we have been able to find concern~\emph{random sinks}~\cite{LeJan_1985,Baxendale_1992}. In addition to memory-loss issues there are works including~\cite{KolyadaSnoha_1996,KolyadaMisiurewiczSnoha_1999,Zhu_etal_2012,Kawan_2013} on the entropy of time-dependent dynamical systems.

Besides the ones mentioned, few limit laws have been proven on the properties of non-random time-dependent dynamical systems. (On the contrary, a vast literature~---~which we cannot even begin to cover here~---~exists on random dynamical systems, concerning both averaged/annealed and quenched limit theorems.) 
In special cases, some central limit theorems have been obtained~\cite{Bakhtin_1994,ConzeRaugi_2007,Nandori_2012}. Indeed, one of the difficulties one faces in this setting is that it is often not even clear how possible results should be formulated, let alone proven. 
 
%\enlargethispage{6mm}
In our proof, a key ingredient is rapid memory loss, which we establish via coupling. In order to identify the limit process, we solve a well-posed martigale problem \cite{StroockVaradhan,RogersWilliams_Vol2,Durrett_1996}. (See Section~\ref{sec:outline} for an outline of the proof). 
In the context of dynamical systems, the idea of resolving to a martingale problem has been used (sparingly) in the theory of averaging:  first, to our knowledge, in~\cite{Dolgopyat_2005}, and then in~\cite{DeSimoiLiverani_2014}, which is closest to our work.
Averaging is a tool in the analysis of so-called slow--fast systems concerning the limit where one of two variables evolves infinitely fast compared to the other: in order to describe the evolution of the \emph{slow} variable, one may be able to ``average out'' the influence of the fast one and thus reduce the problem to an effective one in which the slow variable alone appears. 
We stress that the abstract setup in Section~\ref{sec:motivation} is situated rather at the other end of the spectrum: it involves studying the fast variable~$x_{n,k}$, which (in the limit $n\to\infty$) describes the state of the observed system~$\fs$ under the influence of the large, slow, system $\fS$. Let us mention, however, that averaging ideas could be used to study the statistical properties of QDSs in certain situations, regarding the normalized Birkhoff sum $n^{-\frac12}S_n(x,t)$ (along with an index) as a slow variable and $x$ as a fast variable.

%%%%%%%%%%%%%%%%%%%%%%%%%%%%%%%

\medskip
\subsection{How the paper is organized}
The model QDS, considered in the rest of the paper, is presented in Section~\ref{sec:model}.
In Section~\ref{sec:results} we state our two main results, Theorems~\ref{thm:mean} and~\ref{thm:fluctuations}, keeping the technical prerequisites down to a minimum. 
In Section~\ref{sec:outline} we attempt to outline the proofs in a non-technical manner. In Section~\ref{sec:notation} we explain frequently used notation and record key definitions. 
We then introduce in Sections~\ref{sec:preliminariesI} and~\ref{sec:preliminariesII} the preliminaries necessary for understanding the proofs of Theorems~\ref{thm:mean} and~\ref{thm:fluctuations}, which are presented in Sections~\ref{sec:mean} and~\ref{sec:fluctuations}, respectively. 
In Section~\ref{subsec:generalizations} we also state generalizations of the main results, whose proofs are outlined in Section~\ref{sec:generalizations_proofs}; the proofs involve minor modifications of the preceding sections, so we only indicate the necessary changes, with the hope that such a choice makes the paper easier to read.

\medskip
\section{The model} \label{sec:model}
In this section we introduce the quasistatic dynamical system to be studied in the rest of the paper. 
Fix $\lambda>1$ and $A_*>0$ once and for all.
Let $\cM$ denote the set of $C^2$ expanding maps $T:\bS\to\bS$ on the circle with the following bounds:
\beqn
\inf T' \ge \lambda
\quad\text{and}\quad
\| T''\|_\infty \le A_*\ , \qquad T\in\cM\ .
\eeqn
The space~$\cM$ is endowed with the metric $d_{C^1}$ defined by 
\beqn
d_{C^1}(T_1,T_2) = \sup_{x\in\bS} d(T_1 x,T_2 x) + \|T_1'-T_2'\|_\infty %+ \|T_1''-T_2''\|_\infty
\eeqn
for $T_1,T_2\in\cM$. Here $d$ is the natural metric on $\bS = \bR/\bZ$.

\medskip
We construct a QDS with state space $\bS$ and system space $\cM$ as follows. First, fix a H\"older continuous curve $\gamma:[0,1]\to \cM$ with exponent $\eta\in(0,1)$.
 Let 
$\bfT$ be a triangular array of maps
$$\bfT= \{T_{n,k}\in\cM\ :\ 0\le k\le n, \ n\ge 1\}$$ 
for which
\beq \label{eq:rate}
\sup_{n\ge 1} n^\eta \sup_{0\le t\le 1} d_{C^1}(T_{n,\round{nt}},\gamma_t) < \infty\  .
\eeq
Clearly $(\bfT, \gamma)$ meets the requirements of Definition~\ref{defn:QDS} of a QDS.
A prototypical example to keep in mind is where $T_{n,k} = \gamma_{kn^{-1}}$, though the maps $T_{n,k}$ are not required to live on the curve $\gamma$.
The convergence rate~\eqref{eq:rate} is chosen to reflect the smoothness of the curve; in particular we have similar bounds on $d_{C^1}(T_{n,k}, \gamma_{kn^{-1}})$ and on $d_{C^1}(\gamma_{kn^{-1}}, \gamma_{(k+1)n^{-1}})$, resulting in the bound 
\beq \label{eq:T_rate}
\sup_{n\ge 1} n^\eta \sup_{0\le k <n} d_{C^1}(T_{n,k}, T_{n,k+1}) < \infty\  .
\eeq
 Stronger convergence would not be a natural assumption.

For a fixed~$n$, the maps~$T_{n,k}$ approximate the curve~$\gamma$, traversing from beginning to end as~$k$ increases from~$1$ to~$n$. Moreover,~$T_{n,\round{nt}}$ tends to the well-defined limit $\gamma_t$ as $n\to\infty$.

\medskip
For future use, we point out that every $T\in\cM$ has a unique invariant probability measure $\hat\mu_T$ equivalent to the Lebesgue measure~$\fm$ on~$\bS$. The measure $\hat\mu_T$ we sometimes call an \emph{SRB~measure}, for Sinai--Ruelle--Bowen. For brevity, we write
\beqn
\hat\mu_{t} = \hat\mu_{\gamma_t}\quad\text{and}\quad \hat\mu_{n,k} = \hat\mu_{T_{n,k}}\ .
\eeqn

%%%%%%%%%%%%%%%%%%%%%%%%%%%%%%%%%%%%%%
%%%%%%%%%%%%%%%%%%%%%%%%%%%%%%%%%%%%%%

%%%%%%%%%%%%%%%%%%%%%%%%%%%%%%%%%%%%%%

\medskip
\section{Results}\label{sec:results}
We begin our study of the quasistatic dynamical system $(\bfT,\gamma)$ introduced in Section~\ref{sec:model}. The goal is to understand statistical properties of the QDS. 
So let $f:\bS\to\bR$ be an observable, and denote
\beq\label{eq:f}
f_{n,k} = f\circ T_{n,k}\circ\dots\circ T_{n,1}\ , \quad 0\le k \le n\ .
\eeq
(Our convention is that $f_{n,0} = f$.) This yields a triangular array of random variables once an initial distribution~$\mu$ on $\bS$ is given. 
We define the functions $S_n:\bS\times [0,1]\to\bR$ by
\beqn
S_n(x,t) = \int_0^{nt} f_{n,\round{s}}(x)\,\rd s = \sum_{k=0}^{\round{nt}-1}  f_{n,k}(x) + \fract{nt}f_{n,\round{nt}}(x)\ , \quad n\ge 1\ .
\eeqn

\medskip
\subsection{The mean}
A natural quantity to study first is the mean
\beq\label{eq:zeta}
\zeta_n(x,t) = n^{-1} S_n(x,t)\ .
\eeq
Given an initial probability measure $\mu$ on $\bS$, each $\zeta_n$ is a random element of $C^0([0,1],\bR)$, whose distribution we denote~$\bfP^\mu_n$. As is customary, we also denote $\hat\mu_t(f) = \int f \, \rd\hat\mu_t$, etc.

\medskip
\begin{thm}\label{thm:mean}
Suppose~$f$ is Lipschitz continuous and~$\mu$ is absolutely continuous. The function $t\mapsto \hat\mu_t(f)$ is continuous. The measures $\bfP^\mu_n$ converge weakly, as $n\to\infty$, to the point mass at $\zeta\in C^0([0,1],\bR)$, where
\beq\label{eq:zeta_limit}
\zeta(t) = \int_0^t \hat\mu_s(f)\,\rd s\ .
\eeq
\end{thm}

\medskip
\begin{remark}
Note that the limit distribution in Theorem~\ref{thm:mean} is independent of $\mu$.
\end{remark}

\medskip
In other words, given an arbitrary initial measure $\mu$ having a density, the stochastic process $\zeta_n$ converges, as $n\to\infty$, to the non-random limit~$\zeta$.

\medskip
\subsection{Fluctuations about the mean}
The mean gives a coarse description of a limit statistical property. 
Deeper insight is obtained by studying the fluctuations at a finer scale, by looking into the statistical properties of~$n^\frac12\zeta_n$ instead of~$\zeta_n$. For this to make sense, a centering is needed.
To this end, we choose a sequence of functions $c_n:\bS\times [0,1]\to\bR$, and define the functions~$\chi_n:\bS\times [0,1]\to\bR:$
\beq\label{eq:chi}
\chi_n(x,t) = n^\frac12\zeta_n(x,t) - n^\frac12 c_n(x,t)
\eeq
quantifying the fluctuations of $n^\frac12\zeta_n$ about $n^\frac12 c_n$. The goal is to describe the statistics of these fluctuations, as $n\to\infty$, by a probabilistic limit law.

The choice of a good centering sequence $(c_n)_{n\ge 1}$ turns out to be a delicate issue. There are two canonical choices, both independent of $x$, namely $c_n(t) = \mu(\zeta_n(\slot,t))$ (where~$\mu$ is the initial measure) and $c_n(t) = \fm(\zeta_n(\slot,t))$ (where~$\fm$ is the Lebesgue measure). The first one of these amounts to centering by force: $\mu(\chi_n(\slot,t)) = 0$ holds for all~$t\in[0,1]$.
It is perhaps surprising that this ``natural'' choice does \emph{not} appear to yield good statistics for~$\chi_n$ unless the initial measure~$\mu$ has a sufficiently regular density. The second choice, centering using the Lebesgue measure instead, works better: in that case we obtain a universal limit law for $\chi_n$ as long as the initial measure~$\mu$ is just \emph{absolutely continuous}. We now introduce the notion of an admissible centering sequence, for which our main result holds:

\medskip
\begin{defn}\label{defn:admissible}
We say that a centering sequence $(c_n)_{n\ge 1}$ of functions $c_n:\bS\times [0,1]\to\bR$ is \emph{admissible} with respect to an initial probability measure~$\mu$, if (i) $c_n(x,\slot)$ is continuous for almost every $x$ w.r.t.~$\mu$ and (ii) $t\mapsto n^\frac12c_n(x,t)-n^\frac12\fm(\zeta_n(\slot,t))$ viewed as a random element of~$C^0([0,1],\bR)$ converges to zero in probability w.r.t.~$\mu$, i.e.,
\beqn
\lim_{n\to\infty}\mu\biggl(\biggl\{x\in\bS\ :\ \sup_{t\in[0,1]}\bigl|n^\frac12c_n(x,t)-n^\frac12\fm(\zeta_n(\slot,t))\bigr|>\delta\biggr\}\biggr)  = 0 
\eeqn
for any $\delta>0$. For brevity, we also say in this case that~$c_n$ is admissible (w.r.t.~$\mu$).
\end{defn}

\medskip
\begin{remark}\label{rem:admissible}
If the centering $c_n$ is independent of~$x$, there is no need to refer to an initial measure: $c_n$ is admissible either with respect to every measure or no measure at all; it is admissible if and only if
\beq\label{eq:admissible_indep_x}
\lim_{n\to\infty}\sup_{t\in[0,1]}\bigl|n^\frac12c_n(t)-n^\frac12\fm(\zeta_n(\slot,t))\bigr| = 0\ .
\eeq
Clearly $c_n(t) = \fm(\zeta_n(\slot,t))$ is admissible. 
Since $\fm(\zeta_n(\slot, t))$ converges to~$\zeta(t)$, so does~$c_n(t)$ for any admissible~$c_n$. Whether $n^{\frac12}\fm(\zeta_n(\slot, t))$ converges to $n^{\frac12} \zeta(t)$ (i.e., whether~$\zeta(t)$ is admissible) seems to depend on the regularity of the curve $\gamma$, see Lemma~\ref{lem:admissible}(ii). 
On the other hand, given a measure~$\nu$ (perhaps different from the initial measure~$\mu$), $c_n(t) = \nu(\zeta_n(\slot,t))$ may or may not be admissible, depending on the regularity of (the density of)~$\nu$. 
\end{remark}

\medskip
The next lemma establishes admissible centering sequences of the preceding kind under different conditions, and sheds light on the role of the function~$\zeta$ defined in~\eqref{eq:zeta_limit}.
\begin{lem}\label{lem:admissible}
(i) If $\eta\in(0,1)$ is arbitrary and~$\nu$ is an arbitrary probability measure having a Lipschitz continuous density, then the centering~$c_n(t) = \nu(\zeta_n(\slot,t))$ is admissible.

\smallskip
\noindent (ii) If $\eta>\frac12$, then the \emph{explicit centering}
\beqn
c_n(t) = \zeta(t)
\eeqn
is admissible. Here~$\zeta$ is the function defined in~\eqref{eq:zeta_limit}.

\smallskip
\noindent (iii) If $\eta\in(0,1)$ is arbitrary and~$\nu$ is an arbitrary absolutely continuous measure,\linebreak then $\nu(\zeta_n(\slot,t))$ tends to $\zeta(t)$, as $n\to\infty$. The convergence is uniform in~$t$.
\end{lem}
Lemma~\ref{lem:admissible} is proven in Section~\ref{sec:admissible}.

\medskip
We are now ready to proceed to the main result of the paper. Given an initial probability measure~$\mu$ on~$\bS$ and an admissible centering sequence~$(c_n)_{n\ge 1}$, each~$\chi_n$ is a random element of~$C^0([0,1],\bR)$, whose distribution we denote~$\bP^\mu_n$. We also denote
\beqn
\hat f_t = f - \hat\mu_t(f)
\eeqn
and
\beq\label{eq:sigma}
\hat\sigma_t^2(f) = \lim_{m\to\infty} \hat\mu_t\!\left[\left(\frac{1}{\sqrt m}\sum_{k=0}^{m-1} \hat f_t \circ \gamma_t^k\right)^2\, \right] \ ,
\eeq
which is just the limit of the variance of $\frac{1}{\sqrt m}\sum_{k=0}^{m-1} f\circ \gamma_t^k$ with respect to the measure~$\hat\mu_t$.

\medskip
\begin{thm}\label{thm:fluctuations}
Suppose~$f$ is Lipschitz continuous,~$\mu$ is absolutely continuous, and~$(c_n)_{n\ge 0}$ is admissible. The function $t\mapsto \hat\sigma_t^2(f)$ is continuous. The measures $\bP^\mu_n$ converge weakly, as $n\to\infty$, to the law of the process
\beqn
\chi(t) = \int_0^t \hat\sigma_s(f)\,\rd W_s\ .
\eeqn
Here $W$ is a standard Brownian motion, and the stochastic integral is to be understood in the sense of It\={o}.
\end{thm}

\medskip
Let us pause to discuss Theorem~\ref{thm:fluctuations} in conjunction with Lemma~\ref{lem:admissible}. Here $\nu$ is an arbitrary measure having a Lipschitz continuous density. 
For every value of the regularity exponent~$\eta\in(0,1)$ of the curve~$\gamma$, the \emph{implicit} centering~$c_n(t)=\nu(\zeta_n(\slot,t))$ is a valid choice; in particular, $c_n(t)=\fm(\zeta_n(\slot,t))$ can always be used, and if the initial measure~$\mu$ has a Lipschitz continuous density (some hypothesis on the density is necessary, see Example~\ref{ex:centering}), also the choice $c_n(t) = \mu(\zeta_n(\slot,t))$ is admissible. 
The special case~$\eta<\frac12$ is interesting from the point of view that~$\gamma$ is allowed to have the degree of regularity of Brownian paths. (With probability one, the path of a Brownian motion is H\"older continuous with any exponent~$<\frac12$.) Thus, we may think in this case that the driving ambient system forces the observed system $\gamma_t$ to perform a Brownian motion in the space~$\cM$. On the other hand, the case~$\eta>\frac12$ of a
  more regular curve~$\gamma$ allows also for the \emph{explicit centering} $c_n(t) = \zeta(t)$ to be used. The reason for this dichotomy is that the order of~$n^\frac12 \nu(\zeta_n(\slot,t)) - n^\frac12\zeta(t)$ appears to be so large as to  contribute to the limit process, if the regularity exponent~$\eta$ of the curve~$\gamma$ is $\frac12$ or less. 
A similar remark concerns the order of~$n^\frac12 \mu(\zeta_n(\slot,t))-n^\frac12 \nu(\zeta_n(\slot,t))$, if the density of~$\mu$ is not sufficiently regular. Let us illustrate the last point with an example:

\medskip
\begin{example}[An inadmissible centering] \label{ex:centering}
    Let us return to the setting of Example~\ref{ex:degenerate}, so $T = T_{n,k} : x \mapsto 2x \pmod 1$ for all $n,k$, and $f(x)= -1$ for $x \in [0,\frac12)$ and $f(x) = +1$ otherwise. Then $\fm(\zeta_n(\slot,t)) = 0$ for all $n, t$. 
    Taking as our centering the zero function, the random paths $\chi_n(\slot)$ converge in distribution to standard Brownian motion, if Lebesgue measure is the initial distribution. 
    Disregarding the (here inconsequential) non-Lipschitz nature of $f$, Theorem~\ref{thm:fluctuations} states\footnote{Theorem~\ref{thm:fluctuations} can be extended to the setup of the example. Alternatively, a Lipschitz approximation of~$f$ can be used to obtain similar results; we skip the straightforward details.} that for \emph{any} absolutely continuous initial distribution, the random paths converge in distribution to the same process, standard Brownian motion. This need not hold if one changes the centering:

    \medskip
    For $j \geq 1$, let $A_j$ denote the interval $[0, 2^{-2^j})$ define the probability measure  
     $\nu_j =2^{2^j} 1_{A_j} \fm$. 
    For~$k \leq 2^j -1$, $T^k(A_j) \subset [0,\frac12)$ so $\nu_j(f\circ T^k) = -1$,
    while for~$k \geq 2^j$, the pushforward $T^k_* \nu_j = \fm$ so $\nu_j(f \circ T^k) = \fm(f) = 0$. Therefore $\nu_j(\zeta_n(\slot, t)) \leq 0$ for all $j, n, t$. 
    Let $K\in \bN$ be large enough that $\varepsilon =\sum_{j\geq K} \frac1{j^2} < \frac12$. 
    Let $\mu$ be the absolutely continuous probability measure 
    $$ \mu = (1-\varepsilon)\fm + \sum_{j=K}^\infty \frac1{j^2} \nu_j\ .$$
    Since $\nu_j(\zeta_n(\slot, t)) \leq 0$ for each $j$, for each $j$ we have
    $$\mu(\zeta_n(\slot, t)) \leq \frac1{j^2}\nu_j(\zeta_n(\slot, t)) \ .$$
    Let $t \in (0,1]$ and take $n$ big enough that $nt > 2^K$. Denote~$j_n = \round{\log_2 \round{nt}}$, so $2^{j_n} \leq \round{nt} < 2^{j_n+1}$. It follows that  $K \leq j_n \leq \log_2 nt$. 
    Now $\nu_{j_n+1}(\zeta_n(\slot, t)) = n^{-1} \int_0^{nt} (-1)\, \rd s  = -t$. 
    Therefore,
    $$
    n^{\frac12} \mu(\zeta_n(\slot, t)) \leq 
    n^{\frac12} \frac{1}{(j_n+1)^2} \nu_{j_n+1}(\zeta_n(\slot, t)) =  
    -\frac{tn^{\frac12}}{(j_n+1)^2}  \leq - \frac{t n^{\frac12}}{(\log_2 nt + 1)^2} \ ,$$
    which tends to $-\infty$ as $n \to \infty$. 
    In particular, if one sets $c_n(t) = \mu(\zeta_n(\slot, t))$, then~$n^{\frac12}c_n$ converges pointwise (and uniformly away from 0) to $-\infty$, while~$n^{\frac12}\fm(\zeta_n(\slot,t)) = 0$. Hence,~$(c_n)_{n\geq1}$ does not form an admissible centering sequence; see~\eqref{eq:admissible_indep_x}.
\end{example}

\medskip
\begin{remark}
(i) Note that the limit distribution in Theorem~\ref{thm:fluctuations} is independent of $\mu$.

\smallskip
\noindent(ii)~The process~$\chi$ is a martingale starting at $0$, and its quadratic variation is
\beqn
[\chi]_t = \int_0^t \hat\sigma_s^2(f)\, \rd s \ .
\eeqn 
It follows~\cite{RogersWilliams_Vol2} that there exists such a standard Brownian motion $\widetilde W$ that
\beqn
\chi(t) = \widetilde W_{[\chi]_t}\ .
\eeqn
In other words, $\chi$ is a Brownian motion up to the non-random time change $t\mapsto[\chi]_t$.

\smallskip
\noindent(iii)~Regarding the possible degeneracy of the process~$\chi$, a result due to Leonov \cite{Leonov_1961} shows that, given $s\in[0,1]$,~$\hat\sigma_s^2(f)=0$ if and only if there exists an~$L^2$ function~$g:\bS\to\bR$ such that the cohomology equation~$\hat f_s = g\circ\gamma_s - g$ holds almost everywhere, and Livschitz (Liv{\v{s}}ic) rigidity theory~\cite{Livsic_1971,Livsic_1972} shows that~$g$ has a H\"older continuous representative (with the same exponent as~$f$) for which the cohomology equation holds everywhere. 
Equivalently, $\sum_{k=0}^{p-1}\hat f_s\circ\gamma_s^k(x) = 0$ for any point~$x$ of any period~$p$ with respect to~$\gamma_s$. Since periodic points are dense, $\hat\sigma_s^2(f)=0$ amounts to a severe restriction on the choice of~$f$. 
Because periodic orbits are stable under perturbations of the map, it moreover follows  that nondegeneracy is an open condition with respect to the map: if~$\hat\sigma_s^2(f)\ne 0$, then~$\hat\sigma_t^2(f)\ne 0$ for all~$t$ sufficiently close to~$s$. (Actually~$\hat\sigma_t^2(f)$ depends continuously on~$t$; see Lemma~\ref{lem:sigma}.)
\end{remark}

\medskip
\subsection{Generalizations}\label{subsec:generalizations} 
Here we discuss two ways to generalize Theorems~\ref{thm:mean} and~\ref{thm:fluctuations}. First we consider vector-valued observables $f:\bS\to\bR^d$, $d>1$. After that we consider curves~$\gamma$ in~$\cM$ with a finite number of discontinuities (jumps).

\medskip
In the case of a vector-valued observable~$f$, we define the vector-valued processes~$\zeta_n$,~$\zeta$ and~$\chi_n$ according to the formulas~\eqref{eq:zeta},~\eqref{eq:zeta_limit} and~\eqref{eq:chi}, respectively. The centering $c_n$ is also vector-valued, but Definition~\ref{defn:admissible} remains otherwise intact. For clarity, we denote the vector components by superindices enclosed in parentheses: $f = (f^{(1)},\dots,f^{(d)})$, etc.

\medskip
\begin{thm}\label{thm:gen1}
Suppose $f:\bS\to\bR^d$, $d>1$, is Lipschitz continuous. Then Theorems~\ref{thm:mean} and~\ref{thm:fluctuations} as well as Lemma~\ref{lem:admissible} continue to hold with the modifications that~$W$ is a $d$-dimensional standard Brownian motion and~$\hat\sigma_t(f)$ is the~$d\times d$ matrix defined as the square root of the covariance matrix
\beqn
\hat\sigma_t^2(f) = \lim_{m\to\infty} \hat\mu_t\!\left[\left(\frac{1}{\sqrt m}\sum_{k=0}^{m-1} \hat f_t \circ \gamma_t^k\right)\otimes \left(\frac{1}{\sqrt m}\sum_{k=0}^{m-1} \hat f_t \circ \gamma_t^k\right)\, \right].
\eeqn 
Here $(v\otimes v)_{ij} = v^{(i)}v^{(j)}$ for $v=(v^{(1)},\dots,v^{(d)})\in\bR^d$.
\end{thm}

Note that~$\zeta_n^{(i)}$ and~$\chi_n^{(i)}$ are exactly the processes appearing in~\eqref{eq:zeta} and~\eqref{eq:chi} for the scalar-valued observable~$f^{(i)}$, $1\le i\le d$.
The generalization here is that Theorem~\ref{thm:gen1} describes the joint behaviour of each of these two sets of~$d$ scalar-valued processes.

\medskip
The next result allows for jump discontinuities:
\begin{thm}\label{thm:gen2}
Let $\{I_1,\dots,I_m\}$ be a finite partition of $[0,1]$ into intervals. Suppose the curve $\gamma:[0,1]\to \cM$ restricted to each $I_i$, $1\le i\le m$, is H\"older continuous with exponent~$\eta\in(0,1)$, having possibly jumps where two intervals meet. Then Theorems~\ref{thm:mean} and~\ref{thm:fluctuations} as well as Lemma~\ref{lem:admissible} continue to hold, as does Theorem~\ref{thm:gen1}. 
\end{thm}

%%%%%%%%%%%%%%%%%%%%%%%%%%%%%%%%%%%%%%
%%%%%%%%%%%%%%%%%%%%%%%%%%%%%%%%%%%%%%

\medskip
\section{Overview of proofs} \label{sec:outline}

We start off with a probability measure $\mu$ on the circle. This allows us to regard $x$ as a random variable with distribution $\mu$, and then 
$x_{n,k} = T_{n,k}\circ\dots\circ T_{n,1}(x)$ as another random variable with distribution $\mu_{n,k}$, the pushforward of $\mu$. These random variables are neither identically nor independently distributed. 

\medskip
	While not identical, we show in Lemma~\ref{lem:combineddensities} that the distributions  of $x_{n,k}$ and $x_{n,k+1}$ (the densities of the corresponding pushforwards) are similar to each other and to the corresponding SRB distribution. 

\medskip
	Compensating for lack of independence, one can still obtain strong decay of correlations. This is shown in Corollary~\ref{cor:memory_bounds} and Lemma~\ref{lem:dec}, following Lemma~\ref{lem:memory_loss} which is obtained using a coupling argument.

\medskip
	In order to get uniform bounds in the aforementioned results, we must assume something about the regularity of the initial density. Being Lipschitz continuous is enough; often we also assume that the initial density is strictly positive as this simplifies the arguments.  

\medskip
	In Section~\ref{sec:preliminariesII}, we study the processes $\zeta_n, \chi_n$ under two assumptions. First, the initial measure $\mu$ should be strictly positive and Lipschitz continuous. Second, the centering should be of the form $c_n(t) = \mu(\zeta_n(\slot,t))$ [using other measures with regular densities in the centering would work, but choosing $\mu$ leads to a slightly easier proof]. 
To emphasize the particular nature of this process, we denote it by $\xi_n$ instead of $\chi_n$.  We control the second moment $\mu[[\xi_n(t+h) - \xi_n(t)]^2]$ by $h$ times the variance at $t$ plus an error term, in Lemma~\ref{lem:variance}. Using regularity and decorrelation of the pushforward densities, we obtain decorrelation estimates for the processes $\zeta_n, \xi_n$ in Lemma~\ref{lem:dec_moments}. Early in the section, the variance $\hat \sigma_t^2(f)$ is shown to depend continuously on $t$. 

\medskip
	Rather than proving Theorem~\ref{thm:mean} directly, it is simpler to first prove the following proposition which assumes a regular initial density, recalling $\bfP^\mu_n$ denotes the distribution of $\zeta_n$ in $C^0([0,1], \bR)$. 

\begin{prop}\label{prop:meanLip}
Suppose~$f$ is Lipschitz continuous and~$\mu$ is absolutely continuous with a strictly positive Lipschitz continuous density. The measures $\bfP^\mu_n$ converge weakly, as $n\to\infty$, to the point mass at $\zeta\in C^0([0,1],\bR)$, where
\beqn
\zeta(t) = \int_0^t \hat\mu_s(f)\,\rd s\ .
\eeqn
\end{prop}

This proposition is proven in Section~\ref{sec:meanLip}, via a Dynkin formula.
Noting that continuity of $t \mapsto \hat\mu_t(f)$ follows immediately from~\eqref{eq:SRB_parameter_regularity}, the assertions of Theorem~\ref{thm:mean} now follow if one can drop the regularity condition on the absolutely continuous measure. This is done  in Section~\ref{sec:meanProof} via a portmanteau argument. 

\medskip
We prove Theorem~\ref{thm:fluctuations} in Section~\ref{sec:fluctuations}. Again, we first prove the following proposition assuming a more regular initial density and special centering, and then use a portmanteau argument together with Lemma~\ref{lem:admissible} to pass to the full theorem. Lemma~\ref{lem:admissible} is proven at the start of Section~\ref{sec:fluctuations}.

\begin{prop}\label{prop:fluctuations}
Suppose~$f$ is Lipschitz continuous, $\mu$ is absolutely continuous with a strictly positive Lipschitz continuous density, and $c_n(t) = \mu(\zeta_n(t))$. 
The measures $\bP^\mu_n$ (distribution of $\xi_n$) converge weakly, as $n\to\infty$, to the law of the process
\beqn
\chi(t) = \int_0^t \hat\sigma_s(f)\,\rd W_s\ .
\eeqn
Here $W$ is a standard Brownian motion, and the stochastic integral is to be understood in the sense of It\={o}.
\end{prop}

The proof of the proposition is presented in Section~\ref{sec:tech}. Tightness of the measures is shown first, allowing one to extract convergent subsequences. Again a Dynkin formula for the limit measure $\bP$ is proven. A rather lengthy argument then implies that $\bP$ solves the martingale problem corresponding to the desired expression for the diffusion $\chi$. By uniqueness of such solutions, the statement of the proposition holds.

\medskip
The generalizations are proven in Section~\ref{sec:generalizations_proofs}.

\medskip

\section{Notation and definitions}\label{sec:notation}
We use $\mathrm{Lip}$ to denote the space of Lipschitz continuous maps (from $\bS$ to $\bR$), $\mathrm{Lip}(h)$ to denote the Lipschitz constant of a function $h$ and $\|h \|_{\mathrm{Lip}} = \mathrm{Lip}(h) + \| h \|_\infty$ to denote its Lipschitz norm. 

\medskip 
The constant function $1 : \bS \to \bR$ takes the value $1$ identically, while if $I \subset \bS$, the characteristic function $1_I$ takes the value $1$ on $I$ and $0$ on $\bS\setminus I$. 

\medskip
Given $T\in\cM$, the transfer operator $\cL_T:L^1(\fm)\to L^1(\fm)$ is defined by
\beqn
\cL_T h(x) = \sum_{y\in T^{-1}\{x\}} \frac{h(y)}{T'(y)}\ .
\eeqn
It describes the evolution of probability densities under the map~$T$: if~$x$ is distributed according to probability density~$\rho$, then~$T(x)$ is distributed according to the probability density~$\cL_T\rho$.
We will often write
\beqn
\cL_t = \cL_{\gamma_t}
\quad\text{and}\quad
\cL_{n,k} = \cL_{T_{n,k}}\ .
\eeqn
If $\cL_i$ is the transfer operator of $T_i$, then the transfer operator of $\cT = T_k\circ\dots\circ T_1$ satisfies $\cL_{\cT} = \cL_{k}\cdots\cL_{1}$. We shall sometimes write $\cT z$ for $\cT(z)$. 

\medskip
In the following section, $\cL$ will stand for the transfer operator of a general map in $\cM$ and $(\cL_i)_{i=1}^\infty$ will be a general sequence of such operators. The constants appearing will not depend on the choice of these operators. We shall routinely use the facts that $\cL$ maps probability densities to probability densities and that
\beqn
|\cL (uv)| \le \cL|uv| \le \|u\|_\infty \cL |v| \ .
\eeqn

\medskip
Recall that~$\hat\mu_t$ is the SRB measure (equivalent to Lebesgue measure~$\fm$) associated to~$\gamma_t$, and that~$\hat\mu_{n,k}\equiv \hat\mu_{T_{n,k}}$ is the measure associated to $T_{n,k}$. Denote by~$\hat\rho_t$ and~$\hat\rho_{n,k}$ the corresponding densities, and by $\hat \rho_T$ the density of $\hat\mu_T$.

\medskip
\begin{remark}
    Since $\cL_T^k 1$ converges in the supremum norm to $\hat \rho_T$ as $k \to \infty$, it will follow from~\eqref{eq:LTn1} that $\hat\rho_T\in \mathrm{Lip}$ with 
\beq\label{eq:Lip_bound}
\sup_{T\in \cM }\|\hat\rho_T\|_{\mathrm{Lip}} < \infty\ .
\eeq
\end{remark}

\medskip
Let $\mu_{n,k}$ denote the pushforward
\beqn
\mu_{n,k} = (T_{n,k}\circ\dots\circ T_{n,1})_*\mu \ ,
\eeqn
where $\mu$ is a probability measure with density~$\rho$. Then the density of $\mu_{n,k}$ is
\beqn
\rho_{n,k} = \cL_{n,k}\cdots\cL_{n,1}\rho \ .
\eeqn

\medskip
Recall that $f_{n,k} = f\circ T_{n,k}\circ\dots\circ T_{n,1}$; see~\eqref{eq:f}. Given an initial measure~$\mu$, we will often encounter the centered versions of such functions, defined as
\beq\label{eq:bar_f}
\bar f_{n,k} = f_{n,k} - \mu(f_{n,k})\ .
\eeq

\medskip
Our main results are formulated in terms of Lipschitz continuous densities on~$\bS$. However, for technical reasons we introduce the following notions of regularity which will be convenient for the proofs.

\medskip
\begin{defn}
Given a point $z\in\bS$, we define the arc $J_z = \bS\setminus\{z\}$. 
We denote by  $|x-y|_z $ the length of the subarc of $J_z$ with endpoints~$x,y$.
We say that a function $g:J_z\to\bR$ is Lipschitz continuous on $J_z$, with constant $L>0$, if 
\beqn
|g(x)-g(y)| \le L|x-y|_z\ , \quad x,y\in J_z \ .
\eeqn
In other words, $g$ is Lipschitz continuous except across $z\in\bS$ when the distance of~$x$ and~$y$ is understood as the length of the arc between the two points not containing~$z$.
\end{defn}

\begin{defn}
Given $L>0$, let $\cD_L$ be the class of all probability densities $\psi:\bS\to\bR$ with the properties that (i)~$\psi>0$ and (ii)~there exists $z\in\bS$ such that $\log \psi$ is Lipschitz continuous on $J_z = \bS\setminus\{z\}$ with constant~$L$. 
\end{defn}

\medskip
For the most part we will work with the class $\cD_L$, and prove the main theorems in the setting of such initial densities first. Only then will the theorems be generalized to include Lipschitz continuous initial densities which are allowed to take the value zero. %For that reason, the current section involves also densities that do not belong to $\cup_{L>0}\cD_L$.

\medskip

\begin{remark}\label{remark:log_Lip}

(i) A function $g:\bS\to\bR$ that is Lipschitz continuous with constant $L$ relative to the standard metric~$d$ of~$\bS$ is, in the above sense, Lipschitz continuous on $J_z$ with the same constant~$L$, for any $z\in\bS$. This follows from $d(x,y)\le |x-y|_z$.

\smallskip
\noindent (ii) Note that if $\psi\in\cD_L$, then $e^{-L}\le \psi \le e^L$ and that $\psi$ itself (without the logarithm) is Lipschitz continuous on~$J_z$ with constant~$Le^L$. These follow from the fact that the probability density $\psi$ must take the value $1$ somewhere on $J_z$.

\smallskip
\noindent (iii) In the opposite direction, if a probabililty density $\psi$ satisfies $\psi\ge c>0$ and is Lipschitz continuous on $J_z$ with constant $L$, then $\psi\in\cD_{c^{-1}L}$. Indeed,
\beqn
|\log\psi(x)-\log\psi(y)| \le c^{-1} |\psi(x)-\psi(y)| \le c^{-1}L |x-y|_z\ , \quad x,y\in J_z\ .
\eeqn

\end{remark}

\medskip
The following lemma states that $\cup_{L>0}\cD_L$ is dense in the set of all probability densities with respect to the $L^1$ norm.
\begin{lem}\label{lem:approx} 
Let $\varphi:\bS\to\bR$ be an arbitrary probability density. Given any $\ve>0$, there exists $\psi\in\cup_{L>0}\cD_L$ such that 
\beqn
\|\varphi-\psi\|_{L^1} \le \ve\ .
\eeqn
\end{lem}
\begin{proof}
Let $\varepsilon >0$. By~\cite[Theorem~3.14]{Rudin_1987}, continuous
functions on~$J_0 = \bS\setminus\{0\}$ are dense in~$L^1$, so we can approximate~$\varphi$ in~$L^1$  by a
non-negative continuous function~$g$ such that~$\| \varphi - g \|_{L^1} \le
\frac\ve6$.  Then~$g_1 = g + \frac\ve6$ is strictly positive. By
the Weierstrass approximation theorem,~$g_1$ can be
approximated by a strictly positive polynomial~$g_2$ such that~$\|g_1-g_2\|_{L^1}\le\frac\ve6$. Collecting,~$\| \varphi -
g_2 \|_{L^1} \le \frac\ve2$, which also implies $|1-\fm(g_2)|\le\frac\ve2$. Note that~$\psi = \fm(g_2)^{-1}g_2\in\cup_{L>0}\cD_L$; see item~(iii) of Remark~\ref{remark:log_Lip}. Moreover, 
\beqn
\|\varphi-\psi\|_{L^1}\le \| \varphi-g_2 \|_{L^1} + \|g_2 - \fm(g_2)^{-1} g_2\|_{L^1} \le \frac\ve2 + |\fm(g_2)-1| \le \ve\ ,
\eeqn
which is the desired bound.
\end{proof}

\medskip
\section{Preliminaries I: densities, composition and decorrelation}\label{sec:preliminariesI}
The conceptually important results of this section are: Lemma~\ref{lem:D_stability}, which gives uniform control of the Lipschitz constants for the densities of pushforwards; Lemma~\ref{lem:combineddensities}, which says that the pushforward density $\rho_{n, \round{nt}}$ is pretty close to the SRB density $\hat \rho_s$, provided the system has been running for a while and provided $s$ is close to $t$; and the exponential decorrelation results of Corollary~\ref{cor:memory_bounds} and Lemma~\ref{lem:dec}. 

\medskip
Recall that $1 < \lambda \leq T'$ and $\|T''\|_\infty \leq A_* < \infty$ for all maps $T \in \cM$. 

\medskip
\begin{lem} \label{lem:T2T}
    Denoting by $\cT$ any composition of a finite sequence of maps from $\cM$, 
\begin{equation}\label{eq:T2T}
    \sup_\cT \sup_{x \in \bS} \left|\frac{\cT''(x)}{\cT'(x)^2}\right| < \infty\ .
    \end{equation}
    Moreover, there exists $C >0$ such that, for all such $\cT$, if $I$ is an arc on which $\cT$ is diffeomorphic, 
    \begin{equation}\label{eq:Tdistn}
    \left| \frac{\cT'(x)}{\cT'(y)} -1 \right| \le C\, |\cT(I_{x,y})|
    \end{equation}
    for all $x,y \in I$, where $I_{x,y}$ denotes the subarc of $I$ connecting~$x$ and~$y$.
    \end{lem}
    
    \medskip
    \begin{proof}
        Suppose $\cT = T_k \circ \cdots \circ T_1$, where $T_j \in \cM$. For $j \geq 2$, write $P_j = T_{j-1} \circ \cdots \circ T_1$ and let $P_1$ be the identity map. Then $$\cT''(x) = \cT'(x)\sum_{j=1}^{k} \frac{T_j''(P_j(x))}{T'_j(P_j(x))} P_j'(x)\ ,$$ 
        so, noting $|P_j'(x)/\cT'(x)| \leq \lambda^{j-1 -k}$,
        $$\left| \frac{\cT''(x)}{\cT'(x)^2}\right| \leq  A_*\sum_{j =0}^{k-1} \lambda^{j-k}\ ,$$ from which \eqref{eq:T2T} follows. 

	By the change-of-variables formula and \eqref{eq:T2T},
	$$
	\left| \log \frac{\cT'(x)}{\cT'(y)} \right| = \biggl| \int_{I_{x,y}} \frac{\cT''}{\cT'} \,\rd\fm \biggr|  \leq \int_{\cT (I_{x,y})} \left\| \frac{ \cT''}{(\cT')^2}\right \|_\infty  \rd \fm \leq C\, |\cT(I_{x,y})|\ . $$
	Now \eqref{eq:Tdistn} follows from the estimate $|u-1| \le e^{|{\log u}|}|{\log u}|$, $u>0$.
\end{proof}

\medskip
\begin{lem}\label{lem:D_stability}
There exists $L_*>1$ such that the following hold:

\medskip
\noindent (i) 
For all $L\ge L_*$ and $T\in\cM$,
\beqn
\cL_T\cD_L \subset \cD_{L} \ .
\eeqn

\medskip
\noindent (ii) Given $L>L_*$, there exists a constant $N(L)$ such that
\beqn
\cL_k\cdots\cL_1\cD_L \subset \cD_{L_*}\ , \quad  k\ge N(L)\ ,
\eeqn
for any sequence of maps $(T_i)\subset\cM$.

\medskip
\noindent (iii) Let $L\ge L_*$ and $z\in\bS$. Suppose $\psi\in\cD_L$ with $|\log\psi(x)-\log\psi(y)|\le L|x-y|_z$ for~$x,y\in J_z$. Let $\cT$ be a composition of $k$  maps from $\cM$. Let $I\subset J_z$ be a subarc such that~$\cT$ maps~$I$ diffeomorphically onto $J_{\cT z}$. Then the conditional probability density $\psi_I = |{\int_I \psi\,\rd\fm}|^{-1}\psi 1_I$ satisfies
\beqn
\cL_\cT(\psi_I) \in \cD_L\ ,
\eeqn
the (possible) jump being located at $\cT z$. If $k\ge N(L)$, then $\cL_\cT(\psi_I) \in \cD_{L_*}$.
\end{lem}

\medskip
\begin{proof}
Let $L>0$ and assume $\psi\in\cD_L$, with $z$ as in the definition above. Let $\cT$ be a composition of $k$ maps from $\cM$, for some $k \geq 1$. We can subdivide $J_z$ into intervals $I_1,\dots,I_j$ such that $\tau_m = \cT|I_m:I_m\to J_{\cT z}$ is smooth and bijective for $1\le m\le j$. Then
\beqn
\cL_\cT\psi = \sum_{m=1}^j \cL_{\tau_m}\psi\ ,
\eeqn
where $\cL_{\tau_m}\psi(x) = \frac{\psi}{\cT'}(\tau_m^{-1}x)$.
Moreover, using \eqref{eq:T2T}, for some $K < \infty$,
\beqn
\begin{split}
& |\log \cL_{\tau_m}\psi (x) - \log \cL_{\tau_m}\psi (y)| 
\\
&\qquad \le |\log \psi(\tau_m^{-1}x) - \log\psi(\tau_m^{-1}y) | + |\log \cT'(\tau_m^{-1}x) - \log \cT'(\tau_m^{-1}y) |
\\
&\qquad \le L|\tau_m^{-1}x-\tau_m^{-1}y|_z + \|\cT''/(\cT')^2\|_\infty|x-y|_{\cT z}
\\
&\qquad \le (\lambda^{-n}L + K)|x-y|_{\cT z} \equiv r\ .
\end{split}
\eeqn
Hence $e^{-r} \cL_{\tau_m}\psi (x) \le  \cL_{\tau_m}\psi (y) \le e^r \cL_{\tau_m}\psi (x) $. Summing over $m$ yields
\beqn
\cL_\cT\psi \in \cD_{\lambda^{-n}L + K} \ , \quad \psi\in\cD_L\ .
\eeqn
Fix $L_* = (1-\lambda^{-1})^{-1}K > K$. Then parts~(i) and~(ii) of the lemma hold. As for part~(iii),  $\cL_\cT(\psi_I) = \cL_{\tau_m}(\psi_I)$, for $I_m = I$.  Since
$$ 
 |\log \cL_{\cT}\psi_I (x) - \log \cL_{\cT}\psi_I (y)| =
 |\log \cL_{\tau_m}\psi (x) - \log \cL_{\tau_m}\psi (y)| \leq r, $$
 and since $\cL_\cT \psi_I$ is a probability density, we obtain part~(iii).
\end{proof}

\medskip
\begin{lem}\label{lem:LY}
There exists $C>0$ such that 
\beq\label{eq:LTn1}
\|\cL_k\cdots \cL_1 1\|_\infty \le C \ ,
\eeq
\beq\label{eq:LY0}
\|\cL_k\cdots \cL_1\|_{C^0\to C^0} \le C \ ,
\eeq
\beq\label{eq:LY1}
\|\cL_k\cdots \cL_1\|_{\mathrm{Lip}\to\mathrm{Lip}} \le C \ ,
\eeq
and
\beq\label{eq:LY}
\|\cL_k\cdots \cL_1 h\|_{\mathrm{Lip}} \le C\lambda^{-k}\|h\|_{\mathrm{Lip}} + C\|h\|_{C^0} \ .
\eeq
\end{lem}

\medskip
\begin{proof}
    Lemma~\ref{lem:D_stability}(i) implies $\cL_k\cdots \cL_1 1 \in \cD_{L_*}$, which implies~\eqref{eq:LTn1}, from which~\eqref{eq:LY0} follows immediately. 
    Now let $h$ be  Lipschitz continuous.  Let $\cT = T_k \circ \cdots \circ T_1$.
    Let $x_{i,-k}$ and $y_{i,-k}$ denote the preimages of $x$ and $y$ under the same ($i$th) branch of the inverse of $\cT$. (The inverse branches are defined relative to the shorter arc connecting $x$ and $y$.)
Recalling~\eqref{eq:Tdistn}, we deduce
\beqn
\begin{split}
| \cL_{\cT}h(x) - \cL_{\cT}h(y) | 
& \le \sum_{i} \left| \frac{h(x_{i,-k})}{\cT'(x_{i,-k})} - \frac{h(y_{i,-k})}{\cT'(y_{i,-k})} \right|
\\
& \le \sum_{i} \left| \frac{h(x_{i,-k}) - h(y_{i,-k})}{\cT'(x_{i,-k})} + \frac{h(y_{i,-k})}{\cT'(y_{i,-k})}\left[\frac{\cT'(y_{i,-k})}{\cT'(x_{i,-k})} - 1\right] \right|
\\
& \le \lambda^{-k}\mathrm{Lip}(h)d(x,y) \sum_i\frac{1}{\cT'(x_{i,-k})} + Cd(x,y)\sum_{i} \frac{|h|(y_{i,-k})}{\cT'(y_{i,-k})}
\\
& \le \left(\lambda^{-k}\mathrm{Lip}(h) \cL_{\cT}1(x) + C\cL_{\cT}|h|(y) \right) d(x,y)
\\
& \le \left(\lambda^{-k}\mathrm{Lip}(h) + C\|h\|_\infty \right) \|\cL_{\cT}1\|_\infty d(x,y) \ .
\end{split}
\eeqn
Noting that $\|\cL_{\cT}1\|_\infty$ is uniformly bounded by~\eqref{eq:LTn1}, we have
\beqn
\mathrm{Lip}(\cL_{\cT} h) \le C(\lambda^{-k}\mathrm{Lip}(h) + \|h\|_\infty) \ ,
\eeqn
which implies~\eqref{eq:LY}. Clearly~\eqref{eq:LY1} is a consequence of~\eqref{eq:LY}. Let us note that, if $h$ is allowed to have a jump at $z$, that is, if $h$ is only assumed to be Lipschitz continuous with constant~$L>0$ say on $J_{z}$, 
\beq 
\label{eq:Liphz}
| \cL_{\cT}h(x) - \cL_{\cT}h(y) | \leq C(\lambda^{-k}L + \|h\|_\infty) |x - y|_{\cT z} \ .
\eeq
\end{proof}

\medskip
\begin{lem}\label{lem:Lip_stability}
Suppose a probability density~$\psi$ is 
Lipschitz continuous on~$J_z$ with constant~$L$ for some~$z\in\bS$. Let $\cT$ be a composition of $k$ maps from $\cM$.  Then $\cL_{\cT} \psi$ is Lipschitz continuous on~$J_{\cT z}$ with constant~$C(1+L)$. Here~$C>0$ is a uniform constant.
\end{lem}

\medskip
\begin{proof}
Observing that $\|\psi\|_\infty\le 1+L$, the result is a consequence of 
 \eqref{eq:Liphz}.
\end{proof}

\medskip
\begin{lem}\label{lem:transfer_pert}
There exists a constant $C>0$ such that
for any $T_1, T_2 \in \cM$,
\beq\label{eq:T1T2_regularity}
\|\cL_1 - \cL_2 \|_{\mathrm{Lip} \to C^0} \le C d_{C^1}(T_1, T_2) \ .
\eeq
In particular,
\beq\label{eq:transfer_parameter_regularity}
\|\cL_t - \cL_s \|_{\mathrm{Lip} \to C^0} \le C|t-s|^\eta \ .
\eeq
Let $T_i, \tilde T_i \in \cM,\ 1\leq i \leq k$ with associated transfer operators $\cL_i, \tilde \cL_i$. Then
\beq\label{eq:telescope_bound}
\|\cL_{k}\cdots\cL_{1} - \tilde \cL_{k}\cdots\tilde\cL_{1}\|_{\mathrm{Lip} \to C^0} \le C\sum_{j=1}^k d_{C^1}(T_j, \tilde T_j)
\eeq
holds uniformly.
\end{lem}

For obtaining perturbation estimates of the type in Lemma~\ref{lem:transfer_pert} it is paramount that the transfer operators act from a space of rather regular functions~(here $\mathrm{Lip}$) to a space of \emph{less} regular functions~(here $C^0$). See \cite{KellerLiverani_1999,GouezelLiverani_2006} for a general theory.

\medskip
\begin{proof}[Proof of Lemma~\ref{lem:transfer_pert}]
    Let $x \in \bS$ and let $B \subset \bS$ be the neighborhood of~$x$ of radius~$\frac14$. 
    Set $w_T = \int_\bS T'\,\rd\fm\in\bN$. 
For each $T \in\cM$ the $w_T$~inverse branches of $T$ on $B$ are well-defined. 
When $d_{C^1}(T_1, T_2) < \frac18$, there is a canonical correspondence between the inverse branches~$T_{1,i}^{-1}$ and~$T_{2,i}^{-1}$,
 $1 \leq i \leq w_{T_1}$, such that $T_{1,i}^{-1}(B)$ and $T_{2,i}^{-1}(B)$ overlap (in particular, $w_{T_1} = w_{T_2})$. 
 Let us denote $y_i = T_{1,i}^{-1}(x)$ and $\tilde y_i = T_{2,i}^{-1}(x)$. Then $T_2(\tilde y_i) = T_1(y_i)$ implies
\beqn
d(y_i,\tilde y_i) \le \lambda^{-1} d(T_2(y_i),T_2(\tilde y_i)) =  \lambda^{-1} d(T_2(y_i),T_1(y_i))\ .
\eeqn
Hence, the identity
\beqn
\begin{split}
(\cL_1-\cL_2)h(x) 
& = \sum_{i=1}^w \left(\frac{h(y_i)}{T_1'(y_i)} - \frac{h(\tilde y_i)}{T_2'(\tilde y_i)}\right)
\\
& = \sum_{i=1}^w \left[ \frac{h(y_i)-h(\tilde y_i)}{T_1'(y_i)} + h(\tilde y_i)  \left(\frac{1}{T_1'(y_i)} - \frac{1}{T_2'(\tilde y_i)}\right)\right]
\\
& = \sum_{i=1}^w \left[ \frac{h(y_i)-h(\tilde y_i)}{T_1'(y_i)} + h(\tilde y_i)  \left(\frac{T_1'(\tilde y_i)-T_1'(y_i)}{T_1'(y_i)T_1'(\tilde y_i)} + \frac{T_2'(\tilde y_i)-T_1'(\tilde y_i)}{T_2'(\tilde y_i)T_1'(\tilde y_i)}\right)\right]
\end{split}
\eeqn
yields, via \eqref{eq:LY0} to obtain the uniform $C$,
\beqn
\|(\cL_1-\cL_2)h\|_\infty \le C\|h\|_{\mathrm{Lip}}d_{C^1}(T_1,T_2) \ ,
\eeqn
provided $d_{C^1}(T_1,T_2) < \frac18$. For the case $d_{C^1}(T_1,T_2) \geq \frac18$, we can use \eqref{eq:LTn1}, giving
\beqn
\|(\cL_1-\cL_2)h\|_\infty \le 2\|h\|_\infty \sup_{T \in \cM} \|\cL_T 1 \|_\infty \leq C \| h \|_\infty\tfrac18 \leq C \|h\|_{\mathrm{Lip}} d_{C^1}(T_1,T_2) \ .
\eeqn
The above bounds prove~\eqref{eq:T1T2_regularity}. By H\"older continuity of $\gamma$,~\eqref{eq:transfer_parameter_regularity} clearly holds. It now follows from the identity
\beqn
\cL_{k}\cdots\cL_{1} - \tilde \cL_{k}\cdots\tilde\cL_{1} = \sum_{j = 1}^k \cL_{k}\cdots\cL_{{j+1}} (\cL_{j}-\tilde\cL_{j})\tilde\cL_{{j-1}}\cdots\tilde\cL_{1} 
\eeqn
and the uniform bounds in~\eqref{eq:LY0} and~\eqref{eq:LY1} that also \eqref{eq:telescope_bound} is satisfied.
\end{proof}

\medskip

\medskip
\begin{lem}\label{lem:memory_loss}
There exists $\vartheta\in(0,1)$ and, given $L>0$, a constant $C>0$ such that
\beqn
\|\cL_k\cdots\cL_1(\psi^1-\psi^2)\|_{L^1} \le C\vartheta^k\ , \quad k\ge 0 
\eeqn
for all $\psi^1,\psi^2\in\cD_L$.
\end{lem}

\medskip
\begin{proof}
Note that
\beqn
\psi\ge e^{-L_*} \ , \quad \psi\in \cD_{L_*} \ .
\eeqn
Moreover, there exists $L_1>0$ such that
\beqn
R(\psi) \equiv \frac{\psi-\frac12 e^{-L_*}}{1-\frac12 e^{-L_*}} \in \cD_{L_1} \ , \quad \psi \in \cD_{L_*}\ .
\eeqn
We can now construct an inductive coupling scheme for exponential convergence: Suppose~$L>0$ is given and $\psi_i\in\cD_L$, $i=1,2$. Denoting $\psi^i_k = \cL_k\cdots\cL_1\psi^i$, we have $\psi^i_{N(L)} \ge e^{-L_*}$ by Lemma~\ref{lem:D_stability}. Therefore,
\beqn
\|\psi^1_{N(L)}-\psi^2_{N(L)}\|_{L^1} = (1-\tfrac12 e^{-L_*}) \|R(\psi^1_{N(L)})-R(\psi^2_{N(L)})\|_{L^1} \ .
\eeqn
Moreover, $R(\psi^i_{N(L)}) \in \cD_{L_1}$, $i=1,2$. Next, we repeat the procedure, treating $R(\psi^i_{N(L)})$ as the initial densities, obtaining another factor of $1-\tfrac12 e^{-L_*}$ after $N(L_1)$ steps, and so on. This yields
\beqn
\|\psi^1_{k}-\psi^2_{k}\|_{L^1} \le 2(1-\tfrac12 e^{-L_*})^{j+1}
\eeqn
for any $k \ge N(L) + jN(L_1)$ and $j\ge 0$. Since the a priori bound $\|\psi^1_{k}-\psi^2_{k}\|_{L^1}\le 2$ holds for $k<N(L)$, the result follows.
\end{proof}

\medskip
\begin{cor}\label{cor:memory_bounds}
Given $L>0$ and $z\in\bS$, let $g:\bS\to\bR$ be a Lipschitz continuous function on~$J_z$ with constant~$L$, and assume that $\fm(g) = 0$.
\beqn
\|\cL_k\cdots\cL_1 g\|_{L^1} \le C\vartheta^k\ , \quad k\ge 0\ .
\eeqn
Here $\vartheta$ is the same constant as in Lemma~\ref{lem:memory_loss}, and $C$ only depends on $L$.
\end{cor}

\medskip
\begin{proof}
Since $\fm(g) = 0$, we have $g \ge -L$. Then the probability density
\beqn
\psi = \frac{g}{1+L} + 1
\eeqn
is Lipschitz continuous on $J_z$ with constant $\frac{L}{1+L}$, and $\psi\ge \frac{1}{1+L}$. By Remark~\ref{remark:log_Lip},~$\psi\in\cD_L$. By Lemma~\ref{lem:memory_loss},
\beqn
\|\cL_k\cdots\cL_1(\psi-1)\|_{L^1} \le C \vartheta^k\ ,
\eeqn
where $C$ is determined by the value of $L$. We get
\beqn
\|\cL_k\cdots\cL_1 g\|_{L^1} \le C (1+L)\vartheta^k 
\eeqn
as claimed.
\end{proof}

Denoting $\hat f_t = f - \fm(\hat\rho_t f)$, Corollary~\ref{cor:memory_bounds}, above, implies that
\beq\label{eq:rho_f_decay}
\|\cL_t^{k}(\hat\rho_{t} \hat f_t)\|_{L^1}\le C\vartheta^k
\eeq
holds uniformly in $t$ and $k$, with the constant depending on $\|f\|_{\mathrm{Lip}}$ only.

\medskip
Corollary~\ref{cor:memory_bounds} is also instrumental for the following lemma.

\medskip
\begin{lem}\label{lem:SRB_parameter_regularity}
    For any $\beta >0$, there exists such a constant $C>0$ that for any $T_1, T_2 \in \cM$,
\beq\label{eq:SRB_T_regularity}
\|\hat\rho_{1} - \hat\rho_2\|_{L^1} \le Cd_{C^1}(T_1, T_2)^{1-\beta}\ ,
\eeq
where $\hat\rho_j$ is the SRB density for $T_j$. 

\smallskip
For any $\eta'<\eta$, there exists such a constant $C>0$ that
\beq\label{eq:SRB_parameter_regularity}
\|\hat\rho_t - \hat\rho_s\|_{L^1} \le C|t-s|^{\eta'}\ .
\eeq
\end{lem}

\medskip
\begin{proof}
We could appeal to the perturbation theory developed in \cite{KellerLiverani_1999} (see especially Corollary~1 and Remark~5 there). However, since we need control of $\hat\rho_1-\hat\rho_2$ in $L^1$ only, we provide an independent argument. Note that $\hat\rho_T = \cL_T^k\hat\rho_T$ for all~$T\in \cM$ and~$k\ge 1$. Since the $C^0$ norm dominates the $L^1$ norm, 
the bounds in \eqref{eq:telescope_bound} and~\eqref{eq:Lip_bound} together with Corollary~\ref{cor:memory_bounds} show that
\beqn
\begin{split}
    \|\hat\rho_1 - \hat\rho_2\|_{L^1} \le  \|\cL_1^k(\hat\rho_1-\hat\rho_2)\|_{L^1} + \|(\cL_1^k-\cL_2^k)\hat\rho_2\|_{L^1} \le C(\vartheta^k + k d_{C^1}(T_1,T_2))
\end{split}
\eeqn
holds uniformly for all $T_1,T_2$ and all $k$. 
Choosing $k = \lceil \log d_{C^1}(T_1,T_2) /\log\vartheta \rceil$ on the right side yields~\eqref{eq:SRB_T_regularity}, from which~\eqref{eq:SRB_parameter_regularity} follows.
\end{proof}

\medskip
\begin{lem}\label{lem:pushforward_vs_SRB}
There exists~$b>0$ such that the following holds. Given~$\eta'<\eta$ and a probability density~$\rho$ that is Lipschitz continuous on~$J_z$ with constant~$L>0$ for some~$z\in\bS$,
\beqn
\|\rho_{n,k} - \hat\rho_{n,k}\|_{L^1} \le Cn^{-\eta'}\ , \quad  k\ge b\log n\ .
\eeqn
The constant~$C>0$ is determined by $\eta'$ and $L$.
\end{lem}

\medskip
\begin{proof}
Since $\cL_{n,k}\hat\rho_{n,k} = \hat\rho_{n,k}$, we have
\beqn
\begin{split}
\rho_{n,k} - \hat\rho_{n,k} & = 
\cL_{n,k}\cdots \cL_{n,k-K+1}(\rho_{n,k-K}-\hat\rho_{n,k})
+ (\cL_{n,k}\cdots \cL_{n,k-K+1}-\cL_{n,k}^{K})\hat\rho_{n,k} 
\end{split}
\eeqn
for any $K<k$.
In order to bound $\rho_{n,k} - \hat\rho_{n,k}$ in $L^1$, note that Lemma~\ref{lem:Lip_stability} and 
Corollary~\ref{cor:memory_bounds} imply
\beqn
\| \cL_{n,k}\cdots \cL_{n,k-K+1}(\rho_{n,k-K}-\hat\rho_{n,k})  \|_{L^1} \le C\vartheta^K \ ,
\eeqn
where $C$ is determined by $L$.
Because $\|\hat\rho_{n,k}\|_{\mathrm{Lip}}$ is uniformly bounded (see~\eqref{eq:Lip_bound}) and because $\|\cdot\|_{L^1}\le\|\cdot\|_{C^0}$,~\eqref{eq:telescope_bound} yields
\beqn
\begin{split}
& \|(\cL_{n,k}\cdots \cL_{n,k-K+1}-\cL_{n,k}^{K})\hat\rho_{n,k}\|_{L^1}
%\\
%\le\ & 
    \le C K \max_{k-K+1\le j\le k} d_{C^1}(T_{n,j}, T_{n,k})  \le CK^{1+\eta}n^{-\eta}\ .
\end{split}
\eeqn
Collecting,
\beqn
\|\rho_{n,k} - \hat\rho_{n,k}\|_{L^1} \le C(\vartheta^K + K^{1+\eta}n^{-\eta})
\eeqn
for all $n$, $K$ and $k> K$. Setting $K = \lceil -\eta\log n/{\log\vartheta} \rceil$, we see that all terms on the right side above are bounded by $Cn^{-\eta'}$, where $C$ is determined by $\eta'$ and $L$.
\end{proof}

\medskip
\begin{lem}\label{lem:combineddensities}
There exists~$b>0$ such that the following holds. Given~$\eta'<\eta$ and a probability density~$\rho$ that is Lipschitz continuous on~$J_z$ with constant~$L>0$ for some~$z\in\bS$,
\beq \label{eq:sdensity}
\|\rho_{n,\round{nt}} - \hat\rho_t\|_{L^1} \le Cn^{-\eta'} 
\eeq
and more generally,
\beq \label{eq:rsdensity}
\|\rho_{n,\round{nt}} - \hat\rho_s\|_{L^1} \le C(n^{-\eta'} + |t-s|^{\eta'})\ , 
\eeq
provided $t \geq b n^{-1} \log n$. 
The constant~$C>0$ is determined by $\eta'$ and $L$.
\end{lem}

\medskip
\begin{proof}
    Recalling~\eqref{eq:rate}, Lemmas~\ref{lem:SRB_parameter_regularity} and~\ref{lem:pushforward_vs_SRB} yield
    \beqn
        \| \rho_{n,\round{nt}} - \hat \rho_{n,\round{nt}} + \hat \rho_{n,\round{nt}} - \hat \rho_t + \hat \rho_t - \hat \rho_s \|_{L^1} \leq C( n^{-\eta'} + |t-s|^{\eta'})\ ,
        \eeqn
        as required. 
\end{proof}

\medskip
Recall that $f_{n,k} = f\circ T_{n,k}\circ\dots\circ T_{n,1}$; see~\eqref{eq:f}. Corollary~\ref{cor:memory_bounds} is key also in the proof of the following decorrelation result; see~\cite{StenlundSulku_2014} for a generalization.

\medskip
\begin{lem}\label{lem:dec}
There exists a constant $\vartheta\in(0,1)$ such that the following holds. Given a probability measure~$\mu$ with a density that is Lipschitz continuous on~$J_z$ with constant~$L>0$ for some~$z\in\bS$; $k\ge 2$; Lipschitz continuous functions~$f^{(1)},\dots,f^{(k)}:\bS\to\bR$; and numbers $0\le t_1<\dots<t_k\le 1$; we have
\beqn
|\mu(f^{(1)}_{n,\round{nt_1}}\cdots f^{(k)}_{n,\round{nt_k}}) - \mu(f^{(1)}_{n,\round{nt_1}}\cdots f^{(m)}_{n,\round{nt_m}}) \, \mu(f^{(m+1)}_{n,\round{nt_{m+1}}}\cdots f^{(k)}_{n,\round{nt_k}})| \le C\vartheta^{n(t_{m+1}-t_m)}
\eeqn
for all $m\in\{1,\dots,k-1\}$ and $n\ge 0$. The constant $C>0$ is determined by~$k$,~$\|f\|_{\mathrm{Lip}}$ and~$L$.
\end{lem}

\medskip
\begin{proof} Let $\rho$ be the density of $\mu$. Let~$z\in\bS$ be such that~$\rho$ is Lipschitz on~$J_z$ with constant~$L$. For the sake of brevity, let us denote
\beqn
F = f^{(1)}_{n,\round{nt_1}}\cdots f^{(m)}_{n,\round{nt_m}} \ , \quad \tilde F = \cL_{n,\round{nt_m}}\cdots\cL_{n,1} (\rho F)
\eeqn
and
\beqn
G = \prod_{j = {m+1}}^k f^{(j)}\circ T_{n,\round{nt_j}}\circ\cdots\circ T_{n,\round{nt_{m+1}}+1}\ .
\eeqn

\medskip
\noindent{\it Claim.}
The Lipschitz constant of $\tilde F-\fm(\tilde F)$ on $J_z$ and the function $G$ are uniformly bounded over all $n$ and $t_j$, $1\le j\le k$. 

\medskip
\noindent Assuming the Claim, it holds true that
\beqn
\begin{split}
\mu(f^{(1)}_{n,\round{nt_1}}\cdots f^{(k)}_{n,\round{nt_k}}) 
& = \fm(\rho F\cdot G\circ T_{n,\round{nt_{m+1}}}\circ\cdots\circ T_{n,1})
\\
& = \fm(\tilde F\cdot G\circ T_{n,\round{nt_{m+1}}}\circ\cdots\circ T_{n,\round{nt_m}+1})
\\
& = \fm(\tilde F)\fm(G\circ T_{n,\round{nt_{m+1}}}\circ\cdots\circ T_{n,\round{nt_m}+1}) + O(\vartheta^{n(t_{m+1}-t_m)})
\\
& = \mu(F)\fm(G\circ T_{n,\round{nt_{m+1}}}\circ\cdots\circ T_{n,\round{nt_m}+1}) + O(\vartheta^{n(t_{m+1}-t_m)})
\\
& = \mu(F)\mu_{n,\round{nt_m}}(G\circ T_{n,\round{nt_{m+1}}}\circ\cdots\circ T_{n,\round{nt_m}+1}) + O(\vartheta^{n(t_{m+1}-t_m)})
\\
& = \mu(F)\mu(G\circ T_{n,\round{nt_{m+1}}}\circ\cdots\circ T_{n,1}) + O(\vartheta^{n(t_{m+1}-t_m)})\ .
\end{split}
\eeqn
%Here the exponential error term is uniform.
The first, fourth and sixth lines use definitions only, and the second line uses duality. The third and fifth lines follow from Corollary~\ref{cor:memory_bounds}; the claim that $G$ is bounded is needed for both, while the uniform Lipschitz bound on $\tilde F-\fm(\tilde F)$ is crucial for the third line. In the fifth line we also used the fact that the density of $\mu_{n,\round{nt_m}}$ has a uniform bound on its Lipschitz constant; see Lemma~\ref{lem:Lip_stability}.

\medskip 
\noindent{\it Proof of Claim.} The boundedness of $G$ is obvious. Write $\cT = T_{n,\round{nt_m}}\circ\cdots\circ T_{n,1}$. 
Since $\rho$ is a probability density, its Lipschitz property implies the upper bound $\rho\le 1+L$. Thus,
\beqn
\|\rho F\|_{\infty} \le (1+L)\prod_{j=1}^m\|f^{(j)}\|_\infty\ .
\eeqn
Given $x,y\in J_z$, we write $x_{i,-n},y_{i,-n}$ for the corresponding preimages under~$\cT$. (In other words, for each $i$ there is an arc $J_{i,-n}$ containing $x_{i,-n}$ and $y_{i,-n}$ such that $\cT(J_{i,-n}) = J_z$.) The trivial bounds $|\rho(x_{i,-n}) - \rho(y_{i,-n})| \le L|x-y|_z$ and 
$
|f^{(j)}_{n,\round{nt_j}}(x_{i,-n}) - f^{(j)}_{n,\round{nt_j}}(x_{i,-n})| \le \mathrm{Lip}(f^{(j)}) |x-y|_z
$
yield
\beqn
|(\rho F)(x_{i,-n}) - (\rho F)(y_{i,-n})| \le |x-y|_z \, (1+2L)\prod_{j=1}^m\|f^{(j)}\|_\mathrm{Lip} \ .
\eeqn
As in the proof of~\eqref{eq:LY},
\beqn
\begin{split}
| \tilde F(x) - \tilde F(y) | 
& \le \sum_{i} \left| \frac{(\rho F)(x_{i,-n})}{\cT'(x_{i,-n})} - \frac{(\rho F)(y_{i,-n})}{\cT'(y_{i,-n})} \right|
\\
& \le \sum_{i} \left| \frac{(\rho F)(x_{i,-n}) - (\rho F)(y_{i,-n})}{\cT'(x_{i,-n})} + \frac{(\rho F)(y_{i,-n})}{\cT'(y_{i,-n})}\left[\frac{\cT'(y_{i,-n})}{\cT'(x_{i,-n})} - 1\right] \right|
\\
& \le C|x-y|_z \, (1+2L)\prod_{j=1}^m\|f^{(j)}\|_\mathrm{Lip}\ ,
\end{split}
\eeqn
which implies the Claim. The proof of Lemma~\ref{lem:dec} is complete.
\end{proof}

%%%%%%%%%%%%%%%%%%%%%%%%%%%%%%%%%%%%%%
%%%%%%%%%%%%%%%%%%%%%%%%%%%%%%%%%%%%%%

%%%%%%%%%%%%%%%%%%%%%%%%%%%%%%%%%%%%%%
%%%%%%%%%%%%%%%%%%%%%%%%%%%%%%%%%%%%%%

\medskip
\section{Preliminaries II: processes $\zeta_n$, $\xi_n$ and $\chi_n$}\label{sec:preliminariesII}

In this section we study the stochastic processes of interest, assuming  the initial measure~$\mu$ has density~$\rho\in\cup_{L>0}\cD_L$. Without mentioning it separately each time, the observable~$f:\bS\to\bR$ is assumed to be Lipschitz continuous in the rest of the paper. For convenience, let us now recall the notations $f_{n,k} = f\circ T_{n,k}\circ\dots\circ T_{n,1}$ and $\bar f_{n,k} = f_{n,k} - \mu(f_{n,k})$ introduced in~\eqref{eq:f} and~\eqref{eq:bar_f}, respectively.

\medskip
\noindent{\bf Convention.}
Recall the definition of $\chi_n$ in~\eqref{eq:chi}. Given the initial measure~$\mu$, we will often consider the special centering $c_n(t) = \mu(\zeta_n(\slot,t))$ which yields $\mu(\chi_n(\slot,t))=0$ for all~$t$. In this case, to emphasize that the initial measure has been chosen and appears in the definition of the process explicitly, {\bf we use the symbol $\xi_n$ instead of $\chi_n$}, i.e.,
we define the functions $\xi_n:\bS\times [0,1]\to\bR:$
\beq\label{eq:xi}
\xi_n(x,t) = n^\frac12\zeta_n(x,t) - n^\frac12\mu(\zeta_n(\slot,t))\ .
\eeq
In other words, given another centering sequence~$(c_n)_{n\ge 1}$, we have the relation 
\beqn%\label{eq:chi_int}
\chi_n(x,t) = \xi_n(x,t) + n^\frac12 \int_0^{t} \mu(f_{n,\round{ns}}) \,\rd s - n^\frac12 c_n(x,t)\ .
\eeqn

\medskip
In practice, it will be convenient to express the processes~$\zeta_n$ and~$\xi_n$ using integral notation: we have
\beq\label{eq:zeta_int}
\zeta_n(x,t) = \int_0^{t} f_{n,\round{ns}}(x)\,\rd s
\eeq
and
\beq\label{eq:xi_int}
\xi_n(x,t) = n^\frac12 \int_0^{t} \bar f_{n,\round{ns}}(x) \,\rd s
\eeq
from~\eqref{eq:zeta} and~\eqref{eq:xi}, respectively.
From here on, we will routinely drop the $x$-dependence from the notation, writing just~$\zeta_n(t)$ instead of~$\zeta_n(x,t)$, etc.
Of course, 
\beq\label{eq:xi_diff}
\xi_n(t_2)-\xi_n(t_1) 
= n^{\frac12} \int_{t_1}^{t_2} \bar f_{n,\round{ns}} \,  \rd s \ . 
\eeq

\medskip
Let us record a useful bound: given any~$\eta'\in(\frac12,\eta)$, $\mu_{n,\round{ns}}(f) = \mu(f_{n,\round{ns}})$ and~\eqref{eq:sdensity} imply
\beq\label{eq:pushforward_SRB_int}
\mu(f_{n,\round{ns}}) - \hat\mu_s(f) 
= O(n^{-\eta'})\ ,\qquad s\ge bn^{-1}\log n\ .
\eeq

\medskip
\subsection{The variance $\hat\sigma_t^2(f)$}
\begin{lem}\label{lem:sigma}
Recall that $\hat f_t = f - \hat\mu_t(f)$. The limit variance in~\eqref{eq:sigma} can be expressed in terms of the series
\beq\label{eq:sigma_hat}
\hat\sigma_t^2(f) = \hat \mu_t[\hat f_t^2] + 2\sum_{k=1}^\infty \hat \mu_t[\hat f_t \hat f_t\circ \gamma_t^k] = \hat \mu_t[\hat f_t^2] + 2\sum_{k=1}^\infty \fm[\hat f_t \cL_t^k(\hat\rho_t \hat f_t)] \ .
\eeq
The map $[0,1]\to\bR_+:t\mapsto\hat\sigma_t^2(f)$ is (uniformly) continuous. 
\end{lem}

\medskip
\begin{proof}
Notice that the two series are equal term by term. By~\eqref{eq:rho_f_decay},
\beqn
\sup_{0\le s\le 1} |\fm[\hat f_s \cL_t^k(\hat\rho_s \hat f_s)]|\le C\vartheta^k \ ,
\eeqn
so the series converge absolutely at an exponential rate. On the other hand, a direct manipulation of~\eqref{eq:sigma} shows that
\beqn
\hat\sigma_t^2(f) = \hat\mu_t(\hat f_t^2) + 2\lim_{m\to\infty} m^{-1}\sum_{k=1}^{m-1}(m-k) \hat\mu_t(\hat f_t \hat f_t\circ\gamma_t^{k})\ ,
\eeqn
which suffices to prove~\eqref{eq:sigma_hat}.

Defining the truncated sum
\beqn
V_{K,t} = \hat \mu_t[\hat f_t^2] + 2\sum_{k=1}^K \fm[\hat f_t \cL_t^k(\hat\rho_t \hat f_t)] \ ,
\eeqn
we have
\beqn
|\hat\sigma_t^2(f) - \hat\sigma_s^2(f)| \le |V_{K,t}-V_{K,s}| + C\vartheta^K 
\eeqn
for all $s\in[0,1]$ and all $K>0$. Given $\ve>0$, we fix $K$ so large that $C\vartheta^K<\ve/2$. Note that $V_{K,t} = [\fm(\hat\rho_t f^2)-\fm(\hat \rho_t f)^2] + 2\sum_{k=1}^K [\fm(f\cL_t^k(\hat\rho_t f))-\fm(\hat \rho_t f)^2]$. Moreover,
\beqn
\begin{split}
& |\fm(f\cL_t^k(\hat\rho_t f)) - \fm(f\cL_s^k(\hat\rho_s f))| 
\\
& \qquad \le \|f\|_\infty \|\cL_t^k(\hat\rho_t f) - \cL_s^k(\hat\rho_s f)\|_{L^1}
\\
& \qquad \le \|f\|_\infty \|(\cL_t^k-\cL_s^k)(\hat\rho_t f)\|_{L^1} + \|f\|_\infty \|\cL_s^k(\hat\rho_t f-\hat\rho_s f)\|_{L^1}
\\
& \qquad \le \|f\|_\infty \|\cL_t^k-\cL_s^k\|_{\mathrm{Lip}\to C^0}\|\hat\rho_t f\|_{\mathrm{Lip}} + \|f\|_\infty \|(\hat\rho_t-\hat\rho_s) f\|_{L^1}
\\
& \qquad \le \|f\|_{\mathrm{Lip}}^2 \left( \|\cL_t^k-\cL_s^k\|_{\mathrm{Lip}\to C^0}\|\hat\rho_t \|_{\mathrm{Lip}} + \|\hat\rho_t-\hat\rho_s\|_{L^1} \right)\ .
\end{split}
\eeqn
By~\eqref{eq:Lip_bound},~\eqref{eq:telescope_bound} and~\eqref{eq:SRB_parameter_regularity}, we see that $ |V_{K,t}-V_{K,s}|<\ve/2$ and $|\hat\sigma_t^2(f) - \hat\sigma_s^2(f)|<\ve$ for any~$s$ such that $|t-s|$ is sufficiently small. This proves continuity.
\end{proof}

\medskip

\subsection{The second moment $\mu\!\left[[\xi_n(t+h)-\xi_n(t)]^2\right]$}

\begin{lem}\label{lem:variance}
Let $L>0$ and $\rho\in\cD_L$. The quantity $\mu\!\left[[\xi_n(t)-\xi_n(s)]^2\right]$ is uniformly bounded over all $t,s\in[0,1]$ and $n\ge 1$. Moreover,
\beq\label{eq:var_limit}
\mu\!\left[[\xi_n(t+h)-\xi_n(t)]^2\right] = \int_t^{t+h} \hat\sigma_s^2(f)\,\rd s +   o(n^{-\frac12})+h\, o(1)\ ,
\eeq
as $n\to\infty$, whenever $0\le t\le t+h\le 1$. The error terms are uniform in~$t$ and~$h$.\footnote{We remark that the exponent in $o(n^{-\frac12})$ is not optimal, but a choice of convenience.} Finally,
\beq\label{eq:var_int}
\int_t^{t+h} \hat\sigma_s^2(f)\,\rd s\ = h\, \hat\sigma_t^2(f) + o(h) \ ,
\eeq
as $h\to 0$. The error term is uniform in~$t$. All of the bounds depend on~$\rho$ through~$L$ only.
\end{lem}

\medskip 
\begin{proof}
Given~$L>0$, the maximum of~$\rho$ is bounded by a constant determined by~$L$.
Because $\mu(\bar f_{n,k})=0$, % and~$\eta'>\frac12$,~\eqref{eq:xi_diff} implies
\beq\label{eq:second_temp1}
\begin{split}
    \mu\!\left[[\xi_n(t+h)-\xi_n(t)]^2\right] &
%     = n\int_t^{t+h}\!\!\!\int_t^{t+h} \mu( \bar f_{n,\round{ns}} \bar f_{n,\round{nr}} ) \,\rd r\,\rd s 
%    \\
%    & \quad + O\!\left(h^2n^{1-2\eta'} + h n^{-\eta'}\log n + n^{-1}(\log n)^2\right)
%    \\ & \quad
     = n\int_t^{t+h}\!\!\!\int_t^{t+h} \mu( \bar f_{n,\round{ns}} \bar f_{n,\round{nr}} ) \,\rd r\,\rd s \ ,
    %\\
    %& = n\int_t^{t+h}\!\!\!\int_t^{t+h} \mu( \bar f_{n,\round{ns}} \bar f_{n,\round{nr}} ) \,\rd r\,\rd s + o(n^{-\frac12}) + h\, o(1)\ .
\end{split}
\eeq
as $n\to\infty$. %The (non-optimal) error terms are uniform in~$t$ and~$h$.
By Lemma~\ref{lem:dec},
\beq\label{eq:pair_bound}
|\mu(\bar f_{n,\round{ns}} \bar f_{n,\round{nr}}) | \le C\vartheta^{n|r-s|}\ .
\eeq
The boundedness claim of the lemma now follows from elementary integration:
\beqn
\begin{split}
n\int_t^{t+h}\!\!\!\int_t^{t+h} \vartheta^{n|r-s|}  \,\rd r\,\rd s 
& = 2n\int_0^{h}\!\!\int_s^{h} \vartheta^{n(r-s)}  \,\rd r\,\rd s
 = \frac{2}{\log\vartheta}\int_0^{h} (\vartheta^{n(h-s)} - 1)\,\rd s \le \frac{2h}{|\!\log\vartheta|} \ .
\end{split}
\eeqn

As to the second claim, 
let~$\kappa\in(0,\frac14)$ be small enough that~$2\kappa < \eta'(1-\kappa)$ and 
set~$a_n = n^{-1+ \kappa}$. 
We shall now show that, in the limit $n\to\infty$, the sole contribution to the double integral $n\int_t^{t+h} \!\! \int_t^{t+h} \mu( \bar f_{n,\round{ns}} \bar f_{n,\round{nr}} ) \,\rd r\,\rd s$ comes from the parallelepiped $P_n = \{(s,r)\in[t,t+h]^2\,:\,\text{$t+2a_n\le s \le t+h-a_n$ and $|r-s|\le a_n$}\}$ about the diagonal. To that end, let 
\beqn
Q_n = \{(s,r)\in[t,t+h]^2\,:\,\text{$|r-s|\le a_n$ and either $s<t+2a_n$ or $s>t+h-a_n$}\}
\eeqn
and let 
%\beqn
$R_n = \{(s,r)\in[t,t+h]^2\,:\,|r-s|> a_n\}$,
%\eeqn
so that $[t,t+h]^2 = P_n \cup Q_n \cup R_n$. Because the area of $Q_n$ is $O(a_n^2)$, 
\beqn
n\iint_{Q_n} \mu( \bar f_{n,\round{ns}} \bar f_{n,\round{nr}} ) \,\rd r\,\rd s = O(na_n^2) = O(n^{-1 + 2\kappa})\ .
\eeqn
On the other hand,~\eqref{eq:pair_bound} yields
\beqn
n\iint_{R_n} \mu( \bar f_{n,\round{ns}} \bar f_{n,\round{nr}} ) \,\rd r\,\rd s = O(n\vartheta^{na_n}) = o(n^{-1})\ .
\eeqn
Thus, only the contribution of $P_n$ is significant:
\beq\label{eq:second_temp2}
n\int_t^{t+h}\!\!\!\int_t^{t+h} \mu( \bar f_{n,\round{ns}} \bar f_{n,\round{nr}} ) \,\rd r\,\rd s = n\int_{t+2a_n}^{t+h-a_n}\!\!\!\int_{s-a_n}^{s+a_n} \mu( \bar f_{n,\round{ns}} \bar f_{n,\round{nr}} ) \,\rd r\,\rd s + O(n a_n^2) \ .
\eeq

Note that on $P_n$, $s-a_n\ge t+a_n\ge a_n\ge bn^{-1}\log n$ for all but finitely many~$n$. (This motivates the odd $2a_n$ in the definition of $P_n$.) 
By~\eqref{eq:rsdensity}, 
\beq \label{eq:sa_n}
\sup_{r\in(s-a_n,s+a_n)} \|\rho_{n,\round{nr}} - \hat\rho_s\|_{L^1} = O(n^{-\eta'}+a_n^{\eta'}) = O(a_n^{\eta'})\ ,
\eeq
so
\beqn
\sup_{r\in(s-a_n,s+a_n)} |\mu(f_{n,\round{nr}}) - \hat\mu_s(f)| =  O(a_n^{\eta'})\ .
\eeqn
This implies
\beq\label{eq:conv_temp}
n \int_{s-a_n}^{s+a_n} \mu( \bar f_{n,\round{ns}} \bar f_{n,\round{nr}} ) \,\rd r = n \int_{s-a_n}^{s+a_n} \mu(f_{n,\round{ns}} f_{n,\round{nr}} ) - \hat\mu_s(f)^2 \,\rd r + O(n a_n^{1 + \eta'})\ .
\eeq
We split the domain of integration $[s-a_n,s+a_n]$ on the right side into two halves. Setting $b_n = \frac{1}{n}(1-\{ns\})$, we have (using~\eqref{eq:sa_n} to pass to the third line)
\beqn
\begin{split}
& n \int_{s}^{s+a_n} \mu(f_{n,\round{ns}} f_{n,\round{nr}} ) \,\rd r
%\\
%=\ &
= n \int_{0}^{a_n} \mu(f_{n,\round{ns}} f_{n,\round{n(s+r)}} ) \,\rd r
\\
=\ & b_n n  \mu_{n,\round{ns}}(f^2) +  n\int_{b_n}^{a_n} \mu_{n,\round{ns}}(f f\circ T_{n,\round{n(s+r)}}\circ\dots\circ T_{n,\round{ns}+1} ) \,\rd r
\\
=\ & b_n n  \hat\mu_{s}(f^2) +  n\int_{b_n}^{a_n} \hat\mu_{s}(f f\circ T_{n,\round{n(s+r)}}\circ\dots\circ T_{n,\round{ns}+1} ) \,\rd r + O(a_n^{\eta'}+n a_n^{1+ \eta'})
\\
=\ & n \int_0^{b_n} \fm(f\hat\rho_{s} f)\, \rd r +  n\int_{b_n}^{a_n} \fm (f \cL_{n,\round{n(s+r)}}\cdots\cL_{n,\round{ns}+1}(\hat\rho_{s} f)  ) \,\rd r + O(na_n^{1+\eta'}) \ .
\end{split}
\eeqn
We can replace $\cL_{n,\round{n(s+r)}}\cdots\cL_{n,\round{ns}+1}$ by $\cL_s^{\round{n(s+r)}-\round{ns}}$ since, recalling~\eqref{eq:telescope_bound} and~\eqref{eq:Lip_bound}, 
\beqn
\begin{split}
& \left\|\cL_{n,\round{n(s+r)}}\cdots\cL_{n,\round{ns}+1}(\hat\rho_{s} f) - \cL_s^{\round{n(s+r)}-\round{ns}} (\hat\rho_{s} f)\right\|_{L^1}
\\
\le\ & \|\cL_{n,\round{n(s+r)}}\cdots\cL_{n,\round{ns}+1}- \cL_s^{\round{n(s+r)}-\round{ns}}\|_{\mathrm{Lip}\to C^0} \|\hat\rho_s f\|_{\mathrm{Lip}}
\\
\le\ & C n a_n^{1+\eta} = O(n a_n^{1 + \eta'})
\end{split}
\eeqn
uniformly for $r\in[0,a_n]$. Hence,
\beqn
\begin{split}
& n \int_{s}^{s+a_n} \mu(f_{n,\round{ns}} f_{n,\round{nr}} ) \,\rd r
%\\
%=\ &
= n\int_{0}^{a_n} \fm (f \cL_s^{\round{n(s+r)}-\round{ns}}(\hat\rho_{s} f)  ) \,\rd r + O(n^2 a_n^{2 + \eta'})\ .
\end{split}
\eeqn
A similar computation, which we leave to the reader, yields
\beqn
\begin{split}
& n \int_{s-a_n}^{s} \mu(f_{n,\round{ns}} f_{n,\round{nr}} ) \,\rd r
%\\
%=\ &
    = n\int_{-a_n}^{0} \fm (f \cL_s^{\round{ns}-\round{n(s+r)}}(\hat\rho_{s} f)  ) \,\rd r + O(n^2 a_n^{2 + \eta'})\ .
\end{split}
\eeqn
By~\eqref{eq:conv_temp}, we have thus shown that
\beqn
\begin{split}
& n \int_{s-a_n}^{s+a_n} \mu(\bar f_{n,\round{ns}} \bar f_{n,\round{nr}} ) \,\rd r 
\\
& = n\int_{-a_n}^{a_n} \fm (\hat f_s \cL_s^{|\round{n(s+r)}-\round{ns}|}(\hat\rho_{s} \hat f_s)  ) \,\rd r + O(n^2a_n^{2 + \eta'})
\\
& = n\int_{-\infty}^{\infty} \fm (\hat f_s \cL_s^{|\round{n(s+r)}-\round{ns}|}(\hat\rho_{s} \hat f_s)  ) \,\rd r + O(\vartheta^{na_n} + n^2 a_n^{2+\eta'})
\\
& = \hat\sigma_s^2(f) + O(n^2a_n^{2+\eta'}) \ .
\end{split}
\eeqn
The second last line follows from $|\fm (\hat f_s \cL_s^{k}(\hat\rho_{s} \hat f_s))|\le C\|\cL_s^{k}(\hat\rho_{s} \hat f_s)\|_{L^1}\le C\vartheta^k$ (see~\eqref{eq:rho_f_decay}) and the last one from~\eqref{eq:sigma_hat}. Recalling~\eqref{eq:second_temp2} and~\eqref{eq:second_temp1}, we obtain
\beqn
\begin{split}
 \mu\!\left[[\xi_n(t+h)-\xi_n(t)]^2\right] & = \int_{t+2a_n}^{t+h-a_n}\hat\sigma_s^2(f)\,\rd s +  O\!\left(n a_n^{2} + hn^2a_n^{2+\eta'}\right) 
\\
& = \int_{t}^{t+h}\hat\sigma_s^2(f)\,\rd s +  O\!\left(a_n + n a_n^2+ h n^2 a_n^{2+\eta'}\right) \\ 
& = \int_{t}^{t+h}\hat\sigma_s^2(f)\,\rd s +  
O\!\left(n^{-1+2\kappa} + h n^{2\kappa -\eta'(1-\kappa)}\right)
\ ,
\end{split}
\eeqn
which implies~\eqref{eq:var_limit}, by choice of $\kappa$.

\medskip
Finally, the function $s\mapsto\hat\sigma^2_s(f)$ is uniformly continuous by Lemma~\ref{lem:sigma}. Hence, it has an increasing modulus of continuity $w:[0,1]\to\bR_+$ such that $|\hat\sigma^2_t(f)-\hat\sigma^2_s(f)|\le w(|t-s|)$ holds for all $s,t\in[0,1]$ and~$\lim_{\delta\to 0} w(\delta) = 0$. Therefore, $|\int_t^{t+h} \hat\sigma_s^2(f)\,\rd s - h\, \hat\sigma_t^2(f)|\le h\,w(h)=o(h)$ as $h\to 0$, uniformly in~$t$.
\end{proof}

\medskip
\subsection{Decorrelation at the process level}
Next, we introduce useful partitions of $\bS$ having the property that $x\mapsto \xi_n(x,t)$ (and $x\mapsto \zeta_n(x,t)$) is nearly constant on each partition element. To that end, fix $z\in\bS$. For any integer $n\ge 1$ and real number $t\in(0,1)$, there exists an induced partition $\cP_{z,n,t} = \{I_{z,n,t,j}\}_{j=1}^{N_{n,t}}$ of the arc $J_z = \bS\setminus\{z\}$ into subarcs $I_{z,n,t,j}$ with the property that the restriction of $T_{n,\round{nt}}\circ\dots\circ T_{n,1}$ to $I_{z,n,t,j}$ is one-to-one and onto $J_{z_{n,t}} = \bS\setminus\{z_{n,t}\}$ for all~$j$. Here $z_{n,t} = T_{n,\round{nt}}\circ\dots\circ T_{n,1}(z)$. It follows from the uniform expansion property that
\beqn
|T_{n,\round{ns}}\circ\dots\circ T_{n,1}(I_{z,n,t,j})| \le C \lambda^{n(s-t)}
\eeqn
for all $s\le t$. Since $f\in \mathrm{Lip}$, this yields the uniform bound
\beq\label{eq:xi_difference}
\left|\xi_n(x,s) - \xi_n(y,s) \right| \le C n^{-\frac12} 
\eeq
for all $x,y\in I_{z,n,t,j}$ and $j$, for all $s\le t$.  Given bounded and Lipschitz continuous functions~$B_1,\dots,B_m$ on $\bR$, real numbers $0\le t_1<\dots<t_m\le t$, and a probability measure~$\mu$ with density~$\rho>0$, integrating the previous bound with respect to~$\rd\mu(y)$ yields the existence of $C$ --- determined by $\prod_{1\le k\le m}\|B_k\|_\mathrm{Lip}$ --- such that
\beq\label{eq:ave}
\left|B_1(\xi_n(x,t_1))\dots B_m(\xi_n(x,t_m)) - \mu_{z,n,t,j}[\, B_1(\xi_n(t_1))\dots B_m(\xi_n(t_m))] \right| \le C n^{-\frac12}
\eeq
for all $x\in I_{z,n,t,j}$ and $j$. Here $\mu_{z,n,t,j}$ denotes the conditional measure~$\frac1{\mu(I_{z,n,t,j})}\mu[1_{I_{z,n,t,j}}\slot]$ .

\medskip
\begin{lem}\label{lem:dec_moments}
Suppose $\rho\in\cup_{L>0}\cD_L$.

\medskip
\noindent (i)
If $A\in C^\infty(\bR)$, then
\beqn
 \mu\!\left[A(\zeta_n(s)) [\zeta_n(t)-\zeta_n(s)]\right] -  \mu\!\left[A(\zeta_n(s))\right] \mu\!\left[\zeta_n(t)-\zeta_n(s)\right] = o(1)
\eeqn
as $n\to\infty$, whenever $0\le s\le t\le 1$.

\medskip
\noindent (ii) If $A\in C^\infty_c(\bR)$ and $q\in\{1,2\}$, then
\beqn
 \mu\!\left[A(\xi_n(s)) [\xi_n(t)-\xi_n(s)]^q\right] -  \mu\!\left[A(\xi_n(s))\right] \mu\!\left[[\xi_n(t)-\xi_n(s)]^q\right] = o(1)
\eeqn
as $n\to\infty$, whenever $0\le s\le t\le 1$.
\end{lem}

\medskip
\begin{proof}
Note that, since $\zeta_n$ is uniformly bounded, it suffices to assume $A\in C^\infty_c(\bR)$ in both parts of the lemma. We only prove part~(ii) concerning~$\xi_n$, and leave the easier, but similar, part~(i) concerning $\zeta_n$ to the reader.

Let $z\in\bS$ be a point such that $\log\rho$ is Lipschitz continuous on $J_z$ with constant $L>0$, and consider the induced partition $\cP_{z,n,s} = \{I_{z,n,s,j}\}_{j=1}^{N_{n,s}}$ relative to this point.
Since $\mu\!\left[[\xi_n(t)-\xi_n(s)]^q\right]$ is uniformly bounded by Lemma~\ref{lem:variance}, a special case of~\eqref{eq:ave} yields
\beqn
\begin{split}
& \mu\!\left[A(\xi_n(s)) [\xi_n(t)-\xi_n(s)]^q\right] 
\\
=\ & \sum_j \mu\!\left[1_{I_{z,n,s,j}} A(\xi_n(s)) [\xi_n(t)-\xi_n(s)]^q\right] 
\\
=\ & \sum_j  \mu_{z,n,s,j}[A(\xi_n(s))]\,\mu \!\left[1_{I_{z,n,s,j}}  [\xi_n(t)-\xi_n(s)]^q\right] + O(n^{-\frac12})
\\
=\ & \sum_j  \mu[1_{I_{z,n,s,j}} A(\xi_n(s))] \, \mu_{z,n,s,j} \!\left[ [\xi_n(t)-\xi_n(s)]^q\right] + O(n^{-\frac12})\ .
\end{split}
\eeqn
 Here $\mu_{z,n,s,j}$ is the measure $\mu$ conditioned on $I_{z,n,s,j}$; let us denote the conditional density by~$\rho_{z,n,s,j}$. To finish the proof, it is enough to show that
\beqn
\max_j \left|\mu_{z,n,s,j}  \!\left[ [\xi_n(t)-\xi_n(s)]^q\right] - \mu \!\left[ [\xi_n(t)-\xi_n(s)]^q\right] \right| = o(1)
\eeqn
as $n\to\infty$.
Fixing $p>\frac12$ arbitrarily, it follows from~\eqref{eq:xi_diff} that 
\beqn
\xi_n(t)-\xi_n(s) 
 = n^{\frac12} \int_{s+n^{-p}}^{t} \bar f_{n,\round{nr}} \,  \rd r + o(1)
\eeqn
holds for any $t$ and $s$. In other words, the lower limit of integration can be slightly increased, and it is sufficient to prove
\beqn
\max_j \left| \mu_{z,n,s,j} \!\left[ \left[n^{\frac12} \int_{s+n^{-p}}^{t} \bar f_{n,\round{nr}} \,  \rd r \right]^q\right] - \mu \!\left[ \left[n^{\frac12} \int_{s+n^{-p}}^{t} \bar f_{n,\round{nr}} \,  \rd r \right]^q\right] \right| = o(1)\ .
\eeqn
The increment $n^{-p}$ facilitates the use of $L^1$ convergence of the pushforwards of~$\rho_{z,n,s,j}$ and~$\rho$. 
Indeed, since $\bar f_{n,\round{nr}}$ is uniformly bounded, it suffices to show that
\beqn
\max_j \|\cL_{n,\round{n(s+n^{-p})}}\cdots \cL_{n,1}(\rho_{z,n,s,j}-\rho)\|_{L^1} = o(n^{-\frac q2})\ .
\eeqn
Without loss of generality, we may assume that $L\ge L_*$. By virtue of Lemma~\ref{lem:D_stability}, both of the densities $\tilde\rho_{z,n,s,j} = \cL_{n,\round{ns}}\cdots \cL_{n,1}\rho_{z,n,s,j}$ and $\rho_{n,\round{ns}} = \cL_{n,\round{ns}}\cdots\cL_{n,1}\rho$ are then in $\cD_L$. Lemma~\ref{lem:memory_loss} now yields
\beqn
\max_j \|\cL_{n,\round{n(s+n^{-p})}}\cdots \cL_{n,\round{ns}+1}(\tilde\rho_{z,n,s,j}-\rho_{n,\round{ns}})\|_{L^1} \le C\vartheta^{n^{1-p}},
\eeqn
where~$C$ depends on~$L$. The proof is complete.
\end{proof}

%%%%%%%%%%%%%%%%%%%%%%%%%%%%%%%%%%%%%%
%%%%%%%%%%%%%%%%%%%%%%%%%%%%%%%%%%%%%%

\medskip
\section{Proof of Theorem~\ref{thm:mean}}\label{sec:mean}

Recall the integral expression of~$\zeta_n$ in~\eqref{eq:zeta_int}.
Given an initial probability measure~$\mu$ on~$\bS$, each~$\zeta_n$ is a random element of $C^0([0,1],\bR)$ with distribution $\bfP^\mu_n$. The corresponding expectation will be denoted by~$\bfE^\mu_n$. 

\medskip
In probability theory, the notion of tightness plays a central r\^ole in obtaining limit laws:
Since $C^0([0,1],\bR)$ is a complete and separable metric space, Prohorov's theorem states that a collection of probability measures on it is tight if and only if the collection is relatively sequentially compact in the topology of weak convergence \cite{Billinsgley_1999}. Hence, a tight sequence of measures is guaranteed to have limit points in the topology of weak convergence.

\medskip
\begin{lem}\label{lem:mean_tightness}
Let the measure $\mu$ be arbitrary. The sequence of measures $(\bfP^\mu_n)_{n\ge 1}$ is tight.
\end{lem}

\medskip
\begin{proof}
Note already that, for $t_1,t_2\in[0,1]$,
\beqn
\begin{split}
\zeta_n(t_2) - \zeta_n(t_1) 
& = \int_{t_1}^{t_2}  f_{n,\round{ns}}\,\rd s\ . 
\end{split}
\eeqn
Accordingly,
$
|\zeta_n(t_2) - \zeta_n(t_1)| \le (t_2-t_1)\| f \|_\infty.
$
In other words, the sequence $(\zeta_n)_{n\ge 1}$ of functions is uniformly Lipschitz and bounded. This suffices for tightness on the classical Wiener space $C^0([0,1],\bR)$.
\end{proof}

\medskip
 Next, we are going to show that the sequence actually has a unique limit, which we are going to identify. The following Dynkin formula will turn out useful in this regard.
In order to formulate it properly, let us introduce the evaluation functionals $\pi_t:C^0([0,1],\bR)\to\bR$, $t\in[0,1]$, defined by
\beqn
\pi_t(\omega) = \omega(t)\ .
\eeqn

\medskip
\begin{lem}\label{lem:Dynkin1}
Let $\rho\in\cup_{L>0}\cD_L$.
Suppose $\bfP$ is the weak limit of a subsequence~$(\bfP^\mu_{n_k})_{k\ge 1}$, and denote by~$\bfE$ the expectation with respect to~$\bfP$. 
For any $A\in C^\infty(\bR)$,
\beq\label{eq:mean_generator}
\frac{d}{dt}\bfE[A\circ\pi_t] = \bfE\!\left[A'\circ\pi_t \right]\hat \mu_t(f)\ .
\eeq
\end{lem}

\medskip
\begin{proof}
Let $A\in C^\infty(\bR)$. Using the uniform Lipschitz continuity and boundedness of~$(\zeta_n)_{n\ge 1}$, we get
\beqn
\begin{split}
& 
A(\zeta_n(t+h)) - A(\zeta_n(t))
= A'(\zeta_n(t)) [\zeta_n(t+h)-\zeta_n(t)] + O(h^2)\ ,
\end{split}
\eeqn
where the error term is uniform. Next, we integrate the above expansion with respect to~$\mu$ and take~$n\to\infty$ along the subsequence~$(n_k)_{k\ge 1}$.
Lemma~\ref{lem:dec_moments} guarantees that
\beqn
\begin{split}
& \mu\!\left[A'(\zeta_n(t))\, [\zeta_n(t+h)-\zeta_n(t)]\right]
- \mu\!\left[A'(\zeta_n(t))\right]\mu\!\left[\zeta_n(t+h)-\zeta_n(t) \right] = o(1)
\end{split}
\eeqn
as $n\to\infty$. For the weak limit~$\bfP$,
\beqn
\lim_{k\to\infty} \mu\!\left[A(\zeta_n(t+h)) - A(\zeta_n(t))\right] = \bfE\!\left[A\circ\pi_{t+h} - A\circ\pi_t\right]
\eeqn
and
\beqn
\lim_{k\to\infty} \mu\!\left[A'(\zeta_{n_k}(t))\right] = \lim_{k\to\infty} \bfE^\mu_{n_k}\!\left[A'\circ\pi_t\right] = \bfE\!\left[A'\circ\pi_t\right].
\eeqn
Recalling~\eqref{eq:pushforward_SRB_int},
\beqn
\lim_{n\to\infty}\mu[\zeta_n(t+h)-\zeta_n(t)] =  \int_{t}^{t+h}  \hat \mu_s(f)\,\rd s\ .
\eeqn 
By Lemma~\ref{lem:SRB_parameter_regularity},
\beqn
\int_{t}^{t+h} \hat \mu_s(f)\,\rd s =  \hat \mu_{t}(f) h + o(h)\ ,
\eeqn
which finishes the proof.
\end{proof}

\medskip

\subsection{Proof of Proposition~\ref{prop:meanLip}}\label{sec:meanLip}
We now show that the weak limit~$\bfP$ of the subsequence~$(\bfP^\mu_{n_k})_{k\ge 1}$
 is the point mass at $\zeta\in C^0([0,1],\bR)$ defined in~\eqref{eq:zeta_limit}.

\medskip
With the aid of~\eqref{eq:mean_generator}, we begin by computing
\beqn
\begin{split}
\frac{d}{dt}\bfE \! \left[|\pi_t - \zeta(t)|^2\right] 
& =  
\frac{d}{dt}\bfE \! \left[\pi_t^2\right] - 2\zeta(t)\frac{d}{dt}\bfE \! \left[\pi_t\right] - 2\bfE \! \left[\pi_t\right]\frac{d}{dt}\zeta(t)  + \frac{d}{dt} |\zeta(t)|^2
\\
& = 
2\bfE\!\left[\pi_t \right]\hat \mu_t(f) - 2\zeta(t)\hat\mu_t(f) - 2\bfE \! \left[\pi_t\right] \hat\mu_t(f)  + 2\zeta(t)\hat\mu_t(f) = 0\ .
\end{split}
\eeqn
Since $\pi_0 = 0$ almost surely with respect to $\bfP$ and $\zeta(0)=0$, we have $\bfE \! \left[|\pi_t - \zeta(t)|^2\right] = 0$ for all $t\in[0,1]$. By Tonelli's theorem,
\beqn
\bfE \!\left[\int_0^1|\pi_t - \zeta(t)|^2\,\rd t \right] = \int_0^1\bfE \! \left[|\pi_t - \zeta(t)|^2\right] \rd t = 0\ .
\eeqn
By the continuity of the paths, this proves the claim.

\medskip
In particular, the limit~$\bfP$ is independent of the initial density $\rho\in\cup_{L>0}\cD_L$, and of the weakly converging subsequence $(\bfP^\mu_{n_k})_{k\ge 1}$.
Thus, we have shown that, for any initial measure~$\mu$ with such a density, the sequence $(\bfP^\mu_n)_{n\ge 1}$ itself converges weakly to $\bfP$, completing the proof of Proposition~\ref{prop:meanLip}. 
\qed

\medskip
\subsection{Proof of Theorem~\ref{thm:mean}}\label{sec:meanProof}
Next, suppose~$\nu$ is an arbitrary absolutely continuous initial probability measure with density~$\psi$. Let $F:C^0([0,1],\bR)\to\bR$ be an arbitrary bounded continuous function, and denote $M=\sup_{\omega\in C^0([0,1],\bR)}|F(\omega)|$. By Lemma~\ref{lem:approx}, given any~$\ve>0$, there exists a measure~$\mu$ with density~$\rho\in\cup_{L>0}\cD_L$ such that
\beqn
\|\psi-\rho\|_{L^1} \le \frac\ve{2M}\ .
\eeqn
By the established weak convergence of $(\bfP^\mu_n)_{n\ge 1}$, there exists an integer $N>0$ such that
\beqn
|\bfE^\mu_n[F] - \bfE[F]| \le \frac\ve2\ , \quad n\ge N \ .
\eeqn
Then
\beqn
\begin{split}
|\bfE^\nu_n[F] - \bfE^\mu_n[F]|
& = \left|\int F(\zeta_n(x,\slot))\, \rd \nu(x) - \int F(\zeta_n(x,\slot))\, \rd \mu(x)\right|
\\
& \le \int | F(\zeta_n(x,\slot)) \, (\psi(x)-\rho(x)) |\, \rd \fm(x)
\\
& \le M \|\psi-\rho\|_{L^1} \le \frac\ve2 \ ,
\end{split}
\eeqn
so that
\beqn
|\bfE^\nu_n[F] - \bfE[F]| \le \ve\ , \quad n\ge N\ .
\eeqn
By the portmanteau theorem, this suffices to show that $(\bfP^\nu_n)_{n\ge 1}$ converges weakly to $\bfP$.

\medskip
The proof of Theorem~\ref{thm:mean} is now complete. \qed

\medskip
\begin{remark}
The last part of the proof implies that the results of this section hold for arbitrary absolutely continuous initial measures.
\end{remark}

%%%%%%%%%%%%%%%%%%%%%%%%%%%%%%%%%%%%
%%%%%%%%%%%%%%%%%%%%%%%%%%%%%%%%%%%%

\medskip
\section{Proofs of Lemma~\ref{lem:admissible} and Theorem~\ref{thm:fluctuations}}\label{sec:fluctuations}

\subsection{Proof of Lemma~\ref{lem:admissible}}\label{sec:admissible}

If $\eta\in(0,1)$ is arbitrary and~$\nu$ is a measure having a Lipschitz continuous density~$\psi$, then Corollary~\ref{cor:memory_bounds} yields
\beqn
\begin{split}
& \bigl|n^\frac12\nu(\zeta_n(\slot,t))-n^\frac12\fm(\zeta_n(\slot,t))\bigr|
 \le n^\frac12 \int_0^t |\nu(f_{n,\round{ns}}) - \fm(f_{n,\round{ns}})| \, \rd s
\\
& \qquad \le  n^\frac12 \int_0^t |\nu_{n,\round{ns}}(f) - \fm_{n,\round{ns}}(f)| \, \rd s
\le  n^\frac12 \int_0^t \left|\int \cL_{n,\round{ns}}\cdots\cL_{n,1}(\psi-1) \, f \,\rd\fm\right|  \rd s
\\
& \qquad \le C n^\frac12 \int_0^t \vartheta^{ns} \, \rd s \le Cn^{-\frac12}
\end{split}
\eeqn
uniformly in $t$. Hence the admissibility condition~\eqref{eq:admissible_indep_x} is satisfied. This proves item~(i).

\medskip
Now assume $\eta>\frac12$ and fix $\eta'\in(\frac12,\eta)$. Recalling~\eqref{eq:pushforward_SRB_int},
\beq\label{eq:explicit_admissible}
\bigl|n^\frac12\zeta(t)-n^\frac12\fm(\zeta_n(\slot,t))\bigr|
\le n^\frac12 \int_0^t |\hat\mu_s(f) - \fm(f_{n,\round{ns}})| \, \rd s
\le C(n^{\frac12-\eta'} + n^{-\frac12}\log n)\ .
\eeq
Again the admissibility condition~\eqref{eq:admissible_indep_x} is satisfied. This proves item~(ii).

\medskip
Finally, let $\ve>0$ be arbitrary, and denote the density of~$\nu$ by~$\psi$. By Lemma~\ref{lem:approx}, there exists a measure $\mu$ with density $\varphi\in\cD_L$ such that $\|\varphi-\psi\|_{L^1} \le \|f\|_\infty^{-1}\frac\ve2$.
Then
\beqn
\sup_{t\in[0,1]} |\nu(\zeta_n(\slot,t)) - \zeta(t)| \le \int_0^1 | \nu(f_{n,\round{ns}}) - \hat\mu_s(f)  |\,  \rd s \le \frac\ve2 +  \int_0^1 | \mu(f_{n,\round{ns}}) - \hat\mu_s(f)  |\,  \rd s\ .
\eeqn
By~\eqref{eq:pushforward_SRB_int}, the last term is bounded by $\frac\ve2$ for all large enough~$n$. This proves item~(iii).

\medskip
This finishes the proof of Lemma~\ref{lem:admissible}.
\qed

\medskip
\subsection{Proof of Proposition~\ref{prop:fluctuations}}\label{sec:tech}
Throughout this section we will assume that the initial measure~$\mu$ is given and use the centering $c_n(t) = \mu(\zeta_n(t))$ for the process~$\chi_n$. Recall the convention from the beginning of Section~\ref{sec:preliminariesII} that in this case we write $\xi_n$ instead of $\chi_n$, in order to stress the explicit role of the initial measure in the centering. Thus~$\xi_n$ has the definition in~\eqref{eq:xi} and the integral expression in~\eqref{eq:xi_int}. Each~$\xi_n$ is a random element of~$C^0([0,1],\bR)$ with distribution~$\bP^\mu_n$. The expectation corresponding to $\bP^\mu_n$ will be denoted by~$\bE^\mu_n$.

\medskip
The rest of the subsection constitutes the proof of Proposition~\ref{prop:fluctuations}. As in the previous section, the first step is to prove tightness. 

\medskip
\begin{lem}\label{lem:xi_tight}
Let $\mu$ be a measure with a density that is Lipschitz continuous on $J_z$ for some $z\in\bS$. The sequence of measures $(\bP^\mu_n)_{n\ge 1}$ is tight. 
\end{lem}

\medskip
\begin{proof}
Note that $\pi_0 = 0$ almost surely with respect to $\bP^\mu_n$ for all~$n\ge 1$. By Kolmogorov's criterion, it is sufficient to find a constant~$K>0$ such that 
\beq\label{eq:Kolmogorov}
\mu\!\left[|\xi_n(t_2)-\xi_n(t_1)|^4\right] = \bE^\mu_n[|\pi_{t_2}-\pi_{t_1}|^4] \le K|t_2-t_1|^2
\eeq
holds for all $t_1,t_2\in[0,1]$ and all $n\ge 1$. By symmetry, the left side can be expressed as
\beqn
\begin{split}
I = 4!\, n^{2} \int_{t_1}^{t_2} \!\!\! \int_{t_1}^{s} \!\! \int_{t_1}^{r}\!\! \int_{t_1}^{u} \mu(\bar f_{n,\round{ns}} \bar f_{n,\round{nr}} \bar f_{n,\round{nu}} \bar f_{n,\round{nv}}) \, \rd v\, \rd u\, \rd r\,\rd s \ .
\end{split}
\eeqn
Observe that $\mu(\bar f_{n,\round{nt}}) = 0$ for all $t$. Hence, for $v \le u\le r\le s$,
\beqn
\begin{split}
& |\mu(\bar f_{n,\round{ns}} \bar f_{n,\round{nr}} \bar f_{n,\round{nu}} \bar f_{n,\round{nv}})|
\le C \min\!\left(\vartheta^{n(s-r)}\ , \vartheta^{n(u-v)}\right)
\le C\vartheta^{\frac n2(s-r)}\vartheta^{\frac n2(u-v) } 
\end{split}
\eeqn 
by Lemma~\ref{lem:dec}.
Assuming $t_1\le t_2$, let us define $I_* = n \int_{t_1}^{t_2} \int_{t_1}^{y} \vartheta^{\frac n2 (y-x)}\, \rd x \, \rd y$. Then
\beqn
I \le 4!\, I_*^2
\eeqn
by the preceding bound. Since
\beqn
\int_{t_1}^y \vartheta^{-\frac n2 x} \,\rd x = \frac{1}{n\log\vartheta^{-\frac12}}\left(\vartheta^{-\frac n2 y} - \vartheta^{-\frac n2 t_1}\right) \le \frac{\vartheta^{-\frac n2 y}}{n\log\vartheta^{-\frac12}} \ ,
\eeqn
we get
\beqn
I_* \le \frac{(t_2-t_1)}{\log\vartheta^{-\frac12}}\ ,
\eeqn
which gives the desired estimate in~\eqref{eq:Kolmogorov}. Hence, $(\bP^\mu_n)_{n\ge 1}$ is tight.
\end{proof}

\medskip
As in the previous section, the next step is to study the weak limit points of the sequence~$(\bP^\mu_n)_{n\ge 1}$, whose existence is guaranteed by tightness. Again, a Dynkin formula will be instrumental, albeit additional work will be required in the present setting. We write $A\in C_c^\infty(\bR)$ if $A\in C^\infty(\bR)$ and $A$ vanishes outside a compact set.

\medskip
\begin{lem}\label{lem:Dynkin2}
Let $\rho\in\cup_{L>0}\cD_L$.
Suppose $\bP$ is the weak limit of a subsequence~$(\bP^\mu_{n_k})_{k\ge 1}$, and denote by~$\bE$ the expectation with respect to~$\bP$. 
For any $A\in C_c^\infty(\bR)$,
\beq\label{eq:Dynkin}
\bE[A\circ\pi_{t}] = \bE[A\circ\pi_0] + \frac12 \int_0^t \bE[A''\circ\pi_s]\, \hat\sigma_s^2(f)\,\rd s \ .
\eeq
\end{lem}

\medskip
\begin{proof}
By Taylor's theorem, there exists such a $u\in\bR$ that
\beqn
\begin{split}
A(\xi_n(t+h)) & = A(\xi_n(t)) + A'(\xi_n(t)) [\xi_n(t+h)-\xi_n(t)] 
\\
& \qquad + \frac12 A''(\xi_n(t)) [\xi_n(t+h)-\xi_n(t)]^2 
+ \frac16 A'''(u)(\xi_n(t+h)-\xi_n(t))^3\ .
\end{split}
\eeqn
Note that
\beqn
\mu\!\left[|\xi_n(t+h)-\xi_n(t)|^3\right] \le C h^{2\cdot\frac{3}{4}} = o(h)
\eeqn
by Jensen's inequality together with~\eqref{eq:Kolmogorov}. Since $A'''$ is bounded, we thus have
\beqn
\begin{split}
\mu\!\left[A(\xi_n(t+h))\right] & = \mu\!\left[A(\xi_n(t))\right] + \mu\!\left[A'(\xi_n(t)) [\xi_n(t+h)-\xi_n(t)] \right]
\\
& \qquad + \frac12 \mu\!\left[A''(\xi_n(t)) [\xi_n(t+h)-\xi_n(t)]^2\right]
+ o(h) \ 
\end{split}
\eeqn
where the error term is uniform in~$n$ and~$t$.

Lemma~\ref{lem:dec_moments} guarantees that, for $q\in\{1,2\}$,
\beqn
\begin{split}
& \mu\!\left[A^{(q)}(\xi_n(t)) [\xi_n(t+h)-\xi_n(t)]^q \right] - \mu\!\left[A^{(q)}(\xi_n(t))\right] \mu\!\left[[\xi_n(t+h)-\xi_n(t)]^q \right] \to 0
\end{split}
\eeqn
as $n\to\infty$. By~\eqref{eq:xi_diff},
\beqn
\mu\!\left[\xi_n(t+h)-\xi_n(t) \right]=0 \ .
\eeqn
Next, note that
\beqn
\lim_{k\to\infty }\mu\!\left[A''(\xi_{n_k}(t))\right] = \bE[A''\circ\pi_t]\ .
\eeqn
By Lemma~\ref{lem:variance},
\beqn
\lim_{n\to\infty}\mu\!\left[[\xi_n(t+h)-\xi_n(t)]^2 \right] = h\,\hat\sigma_t^2(f) + o(h)\ .
\eeqn
Recall that the function $t\mapsto \hat\sigma_t^2(f)$ is continuous. Since also $t\mapsto \bE[A''\circ\pi_t]$ is continuous, we thus arrive at~\eqref{eq:Dynkin}.
\end{proof}

\medskip
Define the differential operator
\beqn
\mathscr{L}_t = \frac12\hat\sigma_t^2\,\frac{d^2}{dx^2} \ .
\eeqn
Note that~$\mathscr{L}_t$ appears on the right side of~\eqref{eq:Dynkin}. Lemma~\ref{lem:Dynkin2} thus leads us to conjecture that the limit process $\chi$ is a diffusion with~$\mathscr{L}_t$ as its generator. That is, $\chi$ should solve the stochastic differential equation 
\beq\label{eq:SDE}
d\chi(t) = \hat\sigma_t(f)\,dW_t \ ,
\eeq
where $W_t$ is a standard Brownian motion. Indeed, It\=o calculus for $\chi$ defined by~\eqref{eq:SDE} yields a Dynkin formula which is of precisely the same form as~\eqref{eq:Dynkin}, with the law of~$\chi$ in place of~$\bP$. We proceed to prove rigorously that the limit process is indeed characterized by~\eqref{eq:SDE}.

\medskip
Let us briefly discuss the solutions to~\eqref{eq:SDE}, which are here always required to start at~$0$. Since the coefficient $\hat\sigma_t(f)$ is bounded in~$t$ and independent of $\chi$, given a Brownian motion, there exists a strong solution (adapted to the filtration generated by the Brownian motion) which has continuous paths and is strongly unique (i.e., its modifications are indistinguishable). 
Moreover, weak solutions to~\eqref{eq:SDE} are unique in law; from here on we denote the associated law by~$Q$. These facts imply that the martingale problem corresponding to the generator $\mathscr{L}_t$ and the starting point $0$ is well posed: 
\medskip
\begin{lem}\label{lem:well_posed}
The measure $Q$ is the \emph{unique} measure with the properties that $Q(\pi_0 = 0) = 1$ and that, for all $A\in C_c^\infty(\bR)$, the process
\beq\label{eq:Mdef}
M_t = A\circ\pi_t - A\circ\pi_0 - \int_0^t \mathscr{L}_sA\circ\pi_s \,\rd s\ , \quad t\in[0,1]\ ,
\eeq
is a martingale with respect to $Q$ and the filtration $(\fF_t)_{0\le t\le 1}$, where $\fF_t$ is the sigma-algebra on $C^0([0,1],\bR)$ generated by $\{\pi_s\, : \, 0\le s\le t\}$.
\end{lem}
\medskip
We refer to \cite{RogersWilliams_Vol2} for the proofs of the above statements. Also~\cite{Durrett_1996} is a helpful text on stochastic analysis.

\medskip
The next result states that also the measure~$\bP$ solves the above martingale problem. Hence, it follows from the lemma above that $\bP = Q$.

\medskip
\begin{prop}\label{prop:martingale}
Suppose the density of the initial measure $\mu$ is in $\cup_{L>0}\cD_L$ and that~$\bP$ is the weak limit of a subsequence $(\bP_{n_k})_{k\ge 1}$. Then, given any $A\in C_c^\infty(\bR)$, the process~$(M_t)_{t\in[0,1]}$ defined in~\eqref{eq:Mdef} is a martingale with respect to $\bP$ and the filtration $(\fF_t)_{0\le t\le 1}$. In particular,
\beqn
\bP = Q\ .
\eeqn
\end{prop}

\medskip
Note the limit~$\bP$ is then independent of the initial density $\rho\in\cup_{L>0}\cD_L$, and of the weakly converging subsequence $(\bP^\mu_{n_k})_{k\ge 1}$.
Thus, for any initial measure~$\mu$ with such a density, the sequence $(\bP^\mu_n)_{n\ge 1}$ itself converges weakly to $\bP$. 
Accordingly, we have identified the limit of $\xi_n$ to be the stochastic process~$\chi$ appearing in~\eqref{eq:SDE}. In particular, once this proposition is proven, we will have proven Proposition~\ref{prop:fluctuations}.

\medskip
\begin{proof}[Proof of Proposition~\ref{prop:martingale}]
Let $L>0$ and $z\in\bS$ be such that the logarithm of the initial density~$\rho$ is Lipschitz continuous on $J_z$ with constant $L$.
Since~$A$ and~$A''$ are bounded, we have $\bE[|M_t|]<\infty$ for all~$t\in[0,1]$. It remains to prove that
\beqn
\bE[M_t-M_r\,|\,\fF_r] = 0
\eeqn
whenever $0\le r\le t\le 1$. The martingale condition above is equivalent to the one that
\beq\label{eq:premartingale}
\bE\!\left[B_1\circ \pi_{t_1}\cdots B_m\circ \pi_{t_m}\, (M_t-M_r) \right] = 0
\eeq
whenever $m\ge 1$; and $B_1,\dots,B_m:\bR\to\bR$ are bounded, Lipschitz continuous functions; and $0<t_1<\dots<t_m\le r<t\le 1$. We now fix such numbers and functions for good.

Let us fix now $q\in(0,\frac12)$, and write $K_n = \round{n^q (t-r)}$ and $\delta_n = (t-r)/K_n$.  Since
\beqn
M_t-M_r = \sum_{k=0}^{K_n-1} (M_{r+(k+1)\delta_n}-M_{r+k\delta_n}) \ ,
\eeqn
equation~\eqref{eq:premartingale} will follow once we establish that
\beqn
K_n\sup_{r\le u\le t-\delta_n}\mu\!\left[B_1(\xi_n(t_1))\cdots B_m(\xi_n(t_m))\, (M_{u+\delta_n}-M_u)(\xi_n)\right] = o(1)\ ,
\eeqn
as $n\to\infty$.

We proceed as in the proof of Lemma~\ref{lem:dec_moments}, resorting to the induced partition $\cP_{z,n,u} = \{I_{z,n,u,j}\}_{j=1}^{N_{n,u}}$ of the arc $J_z = \bS\setminus\{z\}$. Let $\hat x_{z,n,u,j}$ denote the midpoint of~$I_{z,n,u,j}$ and $c_{z,n,u,j} = \xi_n(\hat x_{z,n,u,j}, u)$ the value of $\xi_n(\slot, u)$ at the midpoint. It will then be convenient to define the function $\hat \xi_{n,u}:\bS\times[0,1]\to\bR$ by setting
\beqn
\hat \xi_{n,u}(x,t) = \xi_n(x,t) - \xi_n(x,u) + c_{z,n,u,j}
\eeqn
for all $x\in I_{z,n,u,j}$ and all $j$. (For brevity, we suppress the $z$-dependence of $\hat\xi_{n,u}$.)  We think of~$\hat\xi_{n,u}$ as a modification of $\xi_n$ according to the value of the latter process at time~$u$. By~\eqref{eq:xi_difference},
\beqn
\sup_{x\in\bS} |\xi_n(x,t)-\hat\xi_{n,u}(x,t)| = \max_j\sup_{x\in I_{z,n,u,j}} |\xi_n(x,u) - \xi_n(\hat x_{z,n,u,j}, u)| \le Cn^{-\frac12}\ ,
\eeqn
uniformly in~$u$.
Since~$A$ and~$A''$ are Lipschitz continuous, the functional $M_{u+\delta_n}-M_u:C^0([0,1],\bR)\to\bR$ satisfies
\beqn
\sup_{x\in\bS} |(M_{u+\delta_n}-M_u)(\xi_n(x,\slot))-(M_{u+\delta_n}-M_u)(\hat\xi_{n,u}(x,\slot))| \le Cn^{-\frac12} \ .
\eeqn
Because $B_1,\dots,B_m$ are furthermore bounded, Lipschitz continous functions,~\eqref{eq:ave} yields
\beqn
\begin{split}
& \mu\!\left[B_1(\xi_n(t_1))\cdots B_m(\xi_n(t_m))\, (M_{u+\delta_n}-M_u)(\xi_n)\right] 
\\
=\ & \mu\!\left[B_1(\xi_n(t_1))\cdots B_m(\xi_n(t_m))\, (M_{u+\delta_n}-M_u)(\hat \xi_{n,u})\right] + O(n^{-\frac12}) 
\\
=\ & \sum_j \mu\!\left[1_{I_{z,n,u,j}} B_1(\xi_n(t_1))\cdots B_m(\xi_n(t_m))\, (M_{u+\delta_n}-M_u)(\hat \xi_{n,u})\right] + O(n^{-\frac12})
\\
=\ & \sum_j  \mu_{z,n,u,j}[B_1(\xi_n(t_1))\cdots B_m(\xi_n(t_m))] \, \mu \!\left[1_{I_{z,n,u,j}}  (M_{u+\delta_n}-M_u)(\hat \xi_{n,u})\right] + O(n^{-\frac12})
\\
=\ & \sum_j  \mu[1_{I_{z,n,u,j}} B_1(\xi_n(t_1))\cdots B_m(\xi_n(t_m))] \, \mu_{z,n,u,j} \!\left[ (M_{u+\delta_n}-M_u)(\hat \xi_{n,u})\right] + O(n^{-\frac12})\ .
\end{split}
\eeqn
Here the error term is again uniform in~$u$. Since $K_n=o(n^\frac12)$, it thus suffices to show that
\beqn
K_n \sup_{r\le u\le t-\delta_n} \max_j \left| \mu_{z,n,u,j} \!\left[ (M_{u+\delta_n}-M_u)(\hat \xi_{n,u}) \right] \right| = o(1)\ .
\eeqn
Here
\beqn
(M_{u+\delta_n}-M_u)(\hat\xi_{n,u}) = A(\hat\xi_{n,u}(u+\delta_n)) - A(\hat\xi_{n,u}(u)) - \int_u^{u+\delta_n} \mathscr{L}_sA(\hat\xi_{n,u}(s))\,\rd s\ .
\eeqn
For $x\in I_{z,n,u,j}$, we Taylor expand $A(\hat\xi_{n,u}(x,u+\delta_n))$ at $\hat\xi_{n,u}(x,u)=c_{z,n,u,j}$. By Taylor's theorem, there exists $\kappa_{z,n,u,j}(x)\in\bR$ such that
\beqn
\begin{split}
(M_{u+\delta_n}-M_{u})(\hat \xi_{n,u})
& = A'(c_{z,n,u,j}) \, [\xi_n(u+\delta_n) - \xi_n(u) ]
\\
& \quad + \left[\frac12 A''(c_{z,n,u,j})\, [\xi_n(u+\delta_n) - \xi_n(u)]^2  - \int_{u}^{u+\delta_n} \mathscr{L}_sA(\hat\xi_{n,u}(s))\,\rd s\right]
\\
& \quad + \frac16 A'''(\kappa_{z,n,u,j}) \, [\xi_n(u+\delta_n) - \xi_n(u)]^3\ .
\end{split}
\eeqn
We bound the term on each line on the right side separately. The second line requires demonstrating cancellations in the difference and is saved for last.

It is instrumental that (assuming $\round{nt_1}\ge N(L)$) the density~$\rho_{z,n,u,j}$ of $\mu_{z,n,u,j}$ satisfies
\beqn
\tilde \rho_{z,n,u,j} = \cL_{n,\round{nu}}\cdots\cL_{n,1}\rho_{z,n,u,j}\in\cD_{L_*}
\eeqn
by part~(iii) of Lemma~\ref{lem:D_stability}. We denote the measure corresponding to~$\tilde \rho_{z,n,u,j}$ by~$\tilde \mu_{z,n,u,j}$. Moreover, it will be convenient to define 
\beqn
\xi_n^{(u)}(h) = n^\frac12 \int_u^{u+h} f_{n,\round{ns},\round{nu}+1} - \mu_{n, \round{nu}} (f_{n,\round{ns},\round{nu}+1}) \,\rd s\ ,
\eeqn
where $f_{n,k,l} = f\circ T_{n,k}\circ\dots\circ T_{n,l}$ for $k\ge l$ and $f_{n,k,k+1}=f$. Then
\beqn
\xi_n(u+h) - \xi_n(u) = \xi_n^{(u)}(h)\circ T_{n,\round{nu}}\circ\dots\circ T_{n,1}\ .
\eeqn
It is helpful to think of $\xi_n^{(u)}$ as $\xi_n$ shifted along the curve~$\gamma$.
%Note that
%\beqn
%\mu_{n,\round{nu}}\!\left[\xi_n^{(u)}(h)\right] = 0\ .
%\eeqn
Below, we will need to change the centering of~$\xi_n^{(u)}$, so we already define
\beqn
\begin{split}
\xi_n^{\nu,(u)}(h) & = n^\frac12 \int_u^{u+h} f_{n,\round{ns},\round{nu}+1} - \nu (f_{n,\round{ns},\round{nu}+1}) \,\rd s
%\\
%& 
= \xi_n^{(u)}(h) - \nu\bigl[\xi_n^{(u)}(h)\bigr] 
\end{split}
\eeqn
for an arbitrary measure~$\nu$ with density~$\psi\in\cD_{L_*}$.
By Corollary~\ref{cor:memory_bounds} applied to the difference~$g = \psi - \rho_{n, \round{nu}}$, we have the uniform bound
\beqn
|\nu(f_{n,\round{ns},\round{nu}+1}) - \mu_{n, \round{nu}} (f_{n,\round{ns},\round{nu}+1})| = O(\vartheta^{\round{ns} - \round{nu}}) = O(\vartheta^{n(s-u)})\ .
\eeqn
Hence, an integration yields
\beqn
\begin{split}
\nu\bigl[\xi_n^{(u)}(h)\bigr] 
& = n^\frac12 \int_{u}^{u+h} \nu(f_{n,\round{ns},\round{nu}+1}) - \mu_{n, \round{nu}} (f_{n,\round{ns},\round{nu}+1}) \,\rd s  
%\\
%& 
= O(n^{-\frac12})
\end{split}
\eeqn
uniformly in~$u$,~$h$ and~$\nu$, so that
\beq\label{eq:xi^u}
\xi_n^{\nu,(u)}(h) =  \xi_n^{(u)}(h) + O(n^{-\frac12})
\eeq
and, by $\nu\bigl[\xi_n^{\nu,(u)}(h)\bigr] = 0$,
\beq\label{eq:xi^u_mean}
\nu\bigl[\xi_n^{(u)}(h)\bigr] = O(n^{-\frac12})
\eeq
uniformly in $u$,~$h$ and $\nu$. Note that the error terms above are independent of~$x$.

\bigskip
\noindent{\it The first term.}
Since
\beqn
\mu_{z,n,u,j}\left [A'(c_{z,n,u,j}) \, [\xi_n(u+\delta_n) - \xi_n(u) ]\right] = A'(c_{z,n,u,j})\, \mu_{z,n,u,j}\!\left [ \xi_n(u+\delta_n) - \xi_n(u) \right]\ ,
\eeqn
where $A'$ is bounded, we only need a bound on the second factor on the right. Here
\beqn
\mu_{z,n,u,j}\!\left [ \xi_n(u+\delta_n) - \xi_n(u) \right] = \tilde \mu_{z,n,u,j}\!\left [ \xi_n^{(u)}(\delta_n) \right] \ .
\eeqn
Recalling the earlier remark on~$\tilde \mu_{z,n,u,j}$,~\eqref{eq:xi^u_mean} yields
\beqn
\tilde \mu_{z,n,u,j}\!\left [ \xi_n^{(u)}(\delta_n) \right] = o(K_n^{-1})\ .
\eeqn
In particular,
\beqn
K_n \sup_{r\le u\le t-\delta_n} \max_j \left| \mu_{z,n,u,j} \!\left[   A'(c_{z,n,u,j}) \, [\xi_n(u+\delta_n) - \xi_n(u) ]    \right] \right| = o(1)\ .
\eeqn

\bigskip
\noindent{\it The third term.}
Using the boundedness of $A'''$ together with Jensen's inequality,
\beqn
\left| \mu_{z,n,u,j} \!\left[ A'''(\kappa_{z,n,u,j}) \, [\xi_n(u+\delta_n) - \xi_n(u)]^3 \right] \right| \le C \mu_{z,n,u,j} \!\left[[\xi_n(u+\delta_n) - \xi_n(u)]^4 \right]^{\frac34}\ .
\eeqn
Here
\beqn
\mu_{z,n,u,j}\!\left [ [\xi_n(u+\delta_n) - \xi_n(u)]^4 \right] = \tilde \mu_{z,n,u,j}\!\left [ \xi_n^{(u)}(\delta_n)^4 \right] \ .
\eeqn
Let us consider an arbitrary measure~$\nu$ with density~$\psi\in\cD_{L_*}$ instead of~$\tilde \mu_{z,n,u,j}$.
First of all, we have the uniform bound
\beqn\label{eq:fourth_moment}
\nu\!\left [ \xi_n^{\nu,(u)}(h)^4 \right] = O(h^{2}) 
\eeqn
analogously to~\eqref{eq:Kolmogorov}. Together with \eqref{eq:xi^u}, Jensen's inequality then shows that
\beq\label{eq:xi^u_fourth}
\nu\!\left [ \xi_n^{(u)}(h)^4 \right] = O(h^2 + h^{\frac32}n^{-\frac12} + hn^{-1} + n^{-2})
\eeq
uniformly. In particular,
\beqn
\tilde \mu_{z,n,u,j}\!\left [ |\xi_n^{(u)}(\delta_n)|^3 \right]  = O(\delta_n^2 + \delta_n^{\frac32}n^{-\frac12} + \delta_nn^{-1} + n^{-2})^\frac34 = o(K_n^{-1})\ ,
\eeqn
so that
\beqn
K_n \sup_{r\le u\le t-\delta_n} \max_j \left| \mu_{z,n,u,j} \!\left[ A'''(\kappa_{z,n,u,j}) \, [\xi_n(u+\delta_n) - \xi_n(u)]^3 \right] \right| = o(1)\ .
\eeqn

\bigskip
\noindent{\it The second term.}
Note that
\beqn
\mu_{z,n,u,j}\!\left [\frac12A''(c_{z,n,u,j}) \, [\xi_n(u+\delta_n) - \xi_n(u) ]^2\right] = \frac12 A''(c_{z,n,u,j})\, \mu_{z,n,u,j}\!\left [ [\xi_n(u+\delta_n) - \xi_n(u)]^2 \right]\ ,
\eeqn
where
\beqn
\mu_{z,n,u,j}\!\left [ [\xi_n(u+\delta_n) - \xi_n(u)]^2 \right] = \tilde \mu_{z,n,u,j}\!\left [ \xi_n^{(u)}(\delta_n)^2 \right] \ .
\eeqn
On the other hand,
\beqn
\begin{split}
& \mu_{z,n,u,j} \!\left[\int_{u}^{u+\delta_n} \mathscr{L}_sA(\hat\xi_{n,u}(s))\,\rd s\right] 
\\
& = \int_{u}^{u+\delta_n}  \mu_{z,n,u,j} \!\left[\frac12\hat\sigma_s^2(f)A''(\hat\xi_{n,u}(s))\right] \rd s
\\
& = \frac12 \int_{u}^{u+\delta_n} \hat\sigma_s^2(f) \, \mu_{z,n,u,j} \!\left[A''(\xi_n(s) - \xi_n(u) + c_{z,n,u,j} )\right] \rd s
\\
& = \frac12 \int_{u}^{u+\delta_n} \hat\sigma_s^2(f) \, \mu_{z,n,u,j} \!\left[A''(\xi_n^{(u)}(s-u)\circ T_{n,\round{nu}}\circ\dots\circ T_{n,1} + c_{z,n,u,j} )\right] \rd s
\\
& = \frac12 \int_{u}^{u+\delta_n} \hat\sigma_s^2(f) \, \tilde\mu_{z,n,u,j} \!\left[A''(\xi_n^{(u)}(s-u) + c_{z,n,u,j} )\right] \rd s\ .
\end{split}
\eeqn
By Taylor's theorem, there exists $\tilde\kappa_{z,n,u,j,s}(x)\in\bR$, such that
\beqn
\tilde\mu_{z,n,u,j} \!\left[A''(\xi_n^{(u)}(s-u) + c_{z,n,u,j} )\right] = A''(c_{z,n,u,j}) + \tilde\mu_{z,n,u,j} \!\left[A'''(\tilde\kappa_{z,n,u,j,s} )\, \xi_n^{(u)}(s-u)\right]\ .
\eeqn
Using the boundedness of $A'''$ together with Jensen's inequality,~\eqref{eq:xi^u_fourth} yields
\beqn
\left| \tilde\mu_{z,n,u,j} \!\left[A'''(\tilde\kappa_{z,n,u,j,s} )\, \xi_n^{(u)}(s-u)\right] \right| \le C \tilde\mu_{z,n,u,j} \!\left[\xi_n^{(u)}(s-u)^4\right]^{\frac14} = O(\delta_n^{\frac12})\ .
\eeqn
Accordingly,
\beqn
\begin{split}
\mu_{z,n,u,j} \!\left[\int_{u}^{u+\delta_n} \mathscr{L}_sA(\hat\xi_{n,u}(s))\,\rd s\right] 
& = \frac12 A''(c_{z,n,u,j})  \int_{u}^{u+\delta_n} \hat\sigma_s^2(f)\, \rd s + O(\delta_n^{\frac32})
\\
& = \frac12 A''(c_{z,n,u,j})\, \hat\sigma_u^2(f)\,\delta_n + o(\delta_n) \ .
\end{split}
\eeqn
In the second line we used~\eqref{eq:var_int}.
We remark that the error term $o(\delta_n) = o(K_n^{-1})$ is uniform in~$u$ and~$j$, and that $A''$ is bounded.
In order to prove that
\beqn
\begin{split}
& K_n \sup_{r\le u\le t-\delta_n} \max_j \left| \mu_{z,n,u,j} \!\left[ \frac12 A''(c_{z,n,u,j})\, [\xi_n(u+\delta_n) - \xi_n(u)]^2  - \int_{u}^{u+\delta_n} \mathscr{L}_sA(\hat\xi_{n,u}(s))\,\rd s\right] \right|
%\\
%& = o(1)\ .
\end{split}
\eeqn
tends to zero as $n\to\infty$, it thus only remains to show
\beqn
K_n \sup_{r\le u\le t-\delta_n} \max_j \left| \tilde \mu_{z,n,u,j}\!\left [ \xi_n^{(u)}(\delta_n)^2 \right] - \hat\sigma_u^2(f)\,\delta_n\right| = o(1)\ .
\eeqn
Let us again consider an arbitrary measure~$\nu$ with density~$\psi\in\cD_{L_*}$ instead of~$\tilde \mu_{z,n,u,j}$. Recalling~\eqref{eq:xi^u} and~$\nu\bigl[\xi_n^{\nu,(u)}(h)\bigr] = 0$, we get the uniform estimate
\beqn
 \nu\!\left [ \xi_n^{(u)}(\delta_n)^2 \right] =  \nu\!\left [ \xi_n^{\nu,(u)}(\delta_n)^2 \right] + O(n^{-1})\ .
\eeqn
Analogously to Lemma~\ref{lem:variance}, 
\beqn
K_n \! \left| \nu\!\left [ \xi_n^{\nu,(u)}(\delta_n)^2 \right] - \hat\sigma_u^2(f)\,\delta_n\right| = K_n  \bigl|o(\delta_n)+o(n^{-\frac12})\bigr| = o(1)\ ,
\eeqn
uniformly in~$u$ and~$\nu$. The last bounds combined yield the desired bound.

\medskip
This finishes the proof of Proposition~\ref{prop:martingale}.
\end{proof}

\medskip
The proof of Proposition~\ref{prop:fluctuations} is now complete.
\qed

\medskip
\subsection{Finishing the proof of Theorem~\ref{thm:fluctuations}}\label{sec:portmanteau} 
It remains to upgrade Proposition~\ref{prop:fluctuations} to the full version of Theorem~\ref{thm:fluctuations}. 
 The upgrade entails relaxing the regularity assumption on the initial measure as well as the assumption that the centering sequence be defined in terms of the initial measure. To facilitate these changes, let us introduce the explicit notation
\beqn
\chi_n^\nu(t) = n^\frac12\zeta_n(t) - n^\frac12\nu(\zeta_n(t))\ ,
\eeqn
for any measure $\nu$, and
\beqn
\chi_n^{c_n}(t) = n^\frac12\zeta_n(t) - n^\frac12 c_n(t)\ ,
\eeqn
for any centering sequence $(c_n)_{n\ge 1}$.
Note that all functions above, including $c_n$, depend on~$x$, but following our earlier convention we suppress it from the notation. Given an initial measure~$\mu$, we denote the laws of~$\chi_n^\nu$ and~$\chi_n^{c_n}$ by~$\bP^{\mu,\nu}_n$ and~$\bP^{\mu,c_n}_n$, respectively. The respective expectations are denoted~$\bE^{\mu,\nu}_n$ and~$\bE^{\mu,c_n}_n$.

Assume~$\mu$ is an arbitrary absolutely continuous measure with density~$\rho$ and $(c_n)_{n\ge 1}$ is an arbitrary centering sequence admissible with respect to~$\mu$; see Definition~\ref{defn:admissible}. Our proof of Theorem~\ref{thm:fluctuations} amounts to showing that there exists a measure~$\nu$ with density~$\psi\in\cup_{L>0}\cD_L$ with the following properties: (i)~$\bP^{\nu,\nu}_n$ approximates $\bP$, (ii)~$\bP^{\nu,\fm}_n$ approximates~$\bP^{\nu,\nu}_n$, (iii)~$\bP^{\mu,\fm}_n$ approximates $\bP^{\nu,\fm}_n$, and (iv)~$\bP^{\mu,c_n}_n$ approximates~$\bP^{\mu,\fm}_n$ arbitrarily well for all large enough~$n\ge 1$. This will be accomplished using the portmanteau theorem, as follows.

\smallskip
Let $F:C^0([0,1],\bR)\to\bR$ be an arbitrary bounded Lipschitz continuous function and~$\ve>0$ an arbitrary number. Denote~$M=\sup_{\omega\in C^0([0,1],\bR)}|F(\omega)|$ and $\ell = \mathrm{Lip}(F)$.

\smallskip
\noindent{\it Step (i).} Observe that Proposition~\ref{prop:fluctuations} applies directly to~$\bP^{\nu,\nu}_n$: for any $\psi\in\cup_{L>0}\cD_L$, there exists an integer $N_1>0$ such that
\beqn
|\bE[F] - \bE^{\nu,\nu}_n[F]| \le \frac\ve4\ , \quad n\ge N_1 \ .
\eeqn

\smallskip
\noindent{\it Step (ii).} Since~$\psi\in\cup_{L>0}\cD_L$, the centering sequence $\nu(\zeta_n(t))$ is admissible with respect to~$\nu$; see Lemma~\ref{lem:admissible}. By Remark~\ref{rem:admissible}, there exists an integer~$N_2>0$ such that $|n^\frac12\nu(\zeta_n(t))-n^\frac12\fm(\zeta_n(t))|\le \frac\ve{4\ell}$ for all $n\ge N_2$. Thus,
\beqn
\begin{split}
|\bE^{\nu,\nu}_n[F] - \bE^{\nu,\fm}_n[F]| 
& = \left|\int F(\chi_n^\nu(x,\slot)) - F(\chi_n^\fm(x,\slot))\, \rd \nu(x)\right|
%\\
%& 
\le \ell \frac\ve{4\ell}
\\
& 
\le \frac\ve4\ , \quad n\ge N_2 \ .
\end{split}
\eeqn

\smallskip
\noindent{\it Step (iii).}
By Lemma~\ref{lem:approx}, we may assume that~$\psi\in\cup_{L>0}\cD_L$ satisfies
\beqn
\|\psi-\rho\|_{L^1} \le \frac\ve{4M}\ .
\eeqn
Then
\beqn
\begin{split}
|\bE^{\nu,\fm}_n[F] - \bE^{\mu,\fm}_n[F]|
& = \left|\int F(\chi_n^\fm(x,\slot))\, \rd \nu(x) - \int F(\chi_n^\fm(x,\slot))\, \rd \mu(x)\right|
\\
& 
\le \int | F(\chi_n^\fm(x,\slot)) \, (\psi-\rho)(x) |\, \rd \fm(x)
\\
& \le M \|\psi-\rho\|_{L^1} \le \frac\ve4 \ , \quad n\ge 1 \ .
\end{split}
\eeqn

\smallskip
\noindent{\it Step (iv).} Recall that the centering sequence $c_n$ is assumed admissible with respect to~$\mu$.
Let us denote $E_n = \{x\in\bS\ :\ \sup_{t\in[0,1]}|n^\frac12c_n(x,t)-n^\frac12\fm(\zeta_n(t))|>\frac\ve{8\ell}\}$. There exists an integer~$N_3>0$ such that~$\mu(E_n) < \frac\ve{16M}$ for all~$n\ge N_3$. Splitting $\int = \int_{E_n} + \int_{\bS\setminus E_n}$, we have
\beqn
\begin{split}
|\bE^{\mu,\fm}_n[F] - \bE^{\mu,c_n}_n[F]| 
& = \left|\int F(\chi_n^\fm(x,\slot)) - F(\chi_n^{c_n}(x,\slot))\, \rd \mu(x)\right|
\\
& 
\le 2M\mu(E_n) + \ell\frac\ve{8\ell}(1-\mu(E_n))
\\
& 
\le \frac\ve4\ , \quad n\ge N_3 \ .
\end{split}
\eeqn

\smallskip
Collecting the bounds, we have shown that
\beqn
|\bE[F] - \bE^{\mu,c_n}_n[F] | \le \ve\ , \quad n\ge \max(N_1,N_2,N_3)\ .
\eeqn
By the portmanteau theorem, this suffices to show that $(\bP^{\mu,c_n}_n)_{n\ge 1}$ converges weakly to $\bP$.

\medskip
The proof of Theorem~\ref{thm:fluctuations} is now complete.
\qed

\medskip
\begin{remark}
The last part of the proof implies that the results of this section hold for arbitrary absolutely continuous initial measures and admissible centering sequences.
\end{remark}

%%%%%%%%%%%%%%%%%%%%%%%%%%%%%%%%%%%%%%%%%%
%%%%%%%%%%%%%%%%%%%%%%%%%%%%%%%%%%%%%%%%%%

\medskip
\section{Proofs of the generalizations}\label{sec:generalizations_proofs}

As mentioned at the beginning, little in the proofs of Theorems~\ref{thm:mean} and~\ref{thm:fluctuations} changes when one passes to their generalizations, Theorems~\ref{thm:gen1} and~\ref{thm:gen2}. In order to keep the presentation as lucid as possible, we have elected to save the generalizations for last. Here we expect the reader to be well familiar with all the preceding sections. Indeed, we will only point the reader to the straightforward adjustments required there to complete the proofs.

\medskip
\subsection{Proof of Theorem~\ref{thm:gen1}}
The first changes occur in Section~\ref{sec:preliminariesII}. In Lemma~\ref{lem:sigma},~$\hat\sigma_t^2$ is now the~$d\times d$ matrix defined in Theorem~\ref{thm:gen1}. Working with vector components, it amounts to a minor modification of the proof of Lemma~\ref{lem:sigma} to show that
\beq\label{eq:sigma_hat_vec}
\begin{split}
\hat\sigma_t^2(f) 
& = \hat \mu_t[\hat f_t\otimes \hat f_t] + \sum_{k=1}^\infty \hat \mu_t[\hat f_t \otimes (\hat f_t\circ \gamma_t^k) + (\hat f_t\circ \gamma_t^k) \otimes \hat f_t] 
\\
& = \hat \mu_t[\hat f_t\otimes \hat f_t] + \sum_{k=1}^\infty \fm[\hat f_t \otimes \cL_t^k(\hat\rho_t \hat f_t) + (\cL_t^k(\hat\rho_t \hat f_t))\otimes \hat f_t] \ ,
\end{split}
\eeq
and that the dependence on~$t$ is continuous. Similarly, we see that Lemma~\ref{lem:variance} remains true, except that we have the $d\times d$~matrix $\mu\!\left[[\xi_n(t)-\xi_n(s)]\otimes [\xi_n(t)-\xi_n(s)]\right]$ in place of the scalar~$\mu\!\left[[\xi_n(t)-\xi_n(s)]^2\right]$. In~\eqref{eq:ave}, the functions~$B_i$ are now bounded and Lipschitz continuous from~$\bR^d$ to~$\bR$, but the bound remains true. 
Lemma~\ref{lem:dec_moments} is modified as follows: we need $A\in C^\infty(\bR^d,\bR)$ in part~(i) and $A\in C_c^\infty(\bR^d,\bR)$ in part~(ii). Then~(i) and~(ii) with~$q=1$ continue to hold componentwise. In the case~$q=2$, the expression $[\xi_n(t)-\xi_n(s)]^2$ is replaced by $[\xi_n(t)-\xi_n(s)]\otimes [\xi_n(t)-\xi_n(s)]$. The proof remains identical.

Let us proceed to Section~\ref{sec:mean}.
In Lemma~\ref{lem:Dynkin1}, we get the Dynkin formula
\beqn
\frac{d}{dt}\bfE[A\circ\pi_t] =  \bfE\!\left[\nabla A\circ\pi_t \right] \cdot \hat \mu_t(f)
\eeqn
in place of~\eqref{eq:mean_generator}, for any $A\in C^\infty(\bR^d,\bR)$. Indeed, the modification of Lemma~\ref{lem:dec_moments} above implies
\beqn
\begin{split}
& \mu\!\left[\nabla A(\zeta_n(t))\cdot [\zeta_n(t+h)-\zeta_n(t)]\right]
- \mu\!\left[\nabla A(\zeta_n(t))\right] \cdot \mu\!\left[\zeta_n(t+h)-\zeta_n(t) \right] = o(1)
\end{split}
\eeqn
as $n\to\infty$; otherwise the proof is similar. The proof of Proposition~\ref{prop:meanLip} remains the same up to passing to vector notation.

Finally, let us point out the changes in Section~\ref{sec:fluctuations}. In the tightness proof of Lemma~\ref{lem:xi_tight}, we have the bound in~\eqref{eq:Kolmogorov} separately for each vector component $\xi_n^{(i)}$, $1\le i\le d$. The vector-valued case follows by an application of the Cauchy--Schwarz inequality. In Lemma~\ref{lem:Dynkin2}, we get, for any $A\in C_c^\infty(\bR^d,\bR)$, the Dynkin formula
\beq\label{eq:Dynkin_vector}
\bE[A\circ\pi_{t}] = \bE[A\circ\pi_0] + \frac12 \sum_{i,j=1}^d\int_0^t \bE[\partial_{ij}^2 A\circ\pi_s]\, \hat\sigma_s^2(f)_{ij}\,\rd s \ ,
\eeq
in place of~\eqref{eq:Dynkin}. This identity is obtained as before, by Taylor expanding $\mu[A(\xi_n(t))]$ and using the modifications of Lemmas~\ref{lem:dec_moments} and~\ref{lem:variance} above. Let us define $\hat\sigma_t(f)$ as the square root of the $d\times d$ covariance matrix $\hat\sigma_t^2(f)$. Then the stochastic differential equation~\eqref{eq:SDE}, where $W_t$ is an $\bR^d$-valued standard Brownian motion, has the partial differential operator
\beqn
\mathscr{L}_t = \frac12\sum_{i,j=1}^d\hat\sigma_t^2(f)_{ij}\,\frac{\partial^2}{\partial x_i\, \partial x_j}
\eeqn
as its generator, so that the Dynkin formula for~$\chi$ has the same form as~\eqref{eq:Dynkin_vector}. Again, the martingale problem corresponding to the generator $\mathscr{L}_t$ and the starting point $0$ is well posed; see Lemma~\ref{lem:well_posed} with the change $A\in C_c^\infty(\bR^d,\bR)$. We refer to~\cite{RogersWilliams_Vol2} for the proofs of these statements. 
Given the above changes, the proof of the martingale property in Proposition~\ref{prop:martingale} generalizes to the vector-valued case in a straightforward manner, which shows that the limit law~$\bP$ of~$\xi_n$ is the one of the process~$\chi$ appearing in~\eqref{eq:SDE}. To complete the proof, note that Section~\ref{sec:portmanteau} continues to apply, mutatis mutandis, switching from the space $C^0([0,1],\bR)$ of real-valued functions to $C^0([0,1],\bR^d)$.
\qed

%%%%%%%%%%%%%%%%%%%%%%%%%%%%%%%%%%%%%

\medskip
\subsection{Proof of Theorem~\ref{thm:gen2}}
Let $0=\hat\tau_0<\hat\tau_1<\cdots<\hat\tau_{m-1}<\hat\tau_m=1$ be the endpoints of the partition elements~$I_i$, $1\le i\le m$. 

The jumps in the curve $\gamma$ cause certain estimates in Section~\ref{sec:preliminariesI} to hold only piecewise, on each partition element. The first change occurs in Lemma~\ref{lem:transfer_pert}: \eqref{eq:transfer_parameter_regularity} continues to hold if~$s$ and~$t$ are in the same partition element. Likewise,~\eqref{eq:SRB_parameter_regularity} in Lemma~\ref{lem:SRB_parameter_regularity} continues to hold in the same piecewise sense. 
Another change occurs in Lemma~\ref{lem:pushforward_vs_SRB}, where the condition on $k$ is replaced by the piecewise analogue $n\hat\tau_i + b\log n \le k < n\hat\tau_{i+1}$ for some~$i$. In Lemma~\ref{lem:combineddensities} the condition on~$t$ is replaced by $\hat\tau_i + bn^{-1}\log n \le t < \hat\tau_{i+1}$ and $s\in I_i$ for some~$i$. 
Otherwise Section~\ref{sec:preliminariesI} remains intact. To recapitulate, the essential change is that the difference $\rho_{n,\round{nt}}-\hat\rho_t$ does not remain small when $t$ passes one of the finitely many singularities $\hat\tau_i$; to regain smallness, it is necessary to wait another $b\log n$ iterates exactly as was the case at $t=0$ before. Moreover, $\hat\rho_t-\hat\rho_s$ is only small if~$t$ and
~$s$ are in the same partition element and close to each other. However, it is important to point out that the singularities do not affect the \emph{regularity bounds} of the densities $\rho_{n,\round{nt}}$ in any way.

In Section~\ref{sec:preliminariesII} there are similar changes:~\eqref{eq:pushforward_SRB_int} holds if $\hat\tau_i + bn^{-1}\log n \le s < \hat\tau_{i+1}$ for some~$i$.
 We also note that the dependence of $\hat\sigma^2_t(f)$ (and of $\hat\mu_t(f)$) on~$t$ is piecewise-continuous, which is inconsequential; see the proof of Lemma~\ref{lem:sigma} (and of Lemma~\ref{lem:transfer_pert}). 
 Lemma~\ref{lem:variance} remains intact: in order to avoid singularities, one excludes in the domain of the $s$-integral on the right side of~\eqref{eq:second_temp2} a neighborhood of radius~$2a_n$ centered at each singularity $\hat\tau_i$, which only results in another error of the same order~$na_n^2$ as before. These are the only parts in Section~\ref{sec:preliminariesII} that require attention.

The above changes affect Section~\ref{sec:mean} in no way, so we are left with Section~\ref{sec:fluctuations}.
There the proof of Lemma~\ref{lem:admissible} stands; in the integral in~\eqref{eq:explicit_admissible} one has to take into account the above change in the condition of~\eqref{eq:pushforward_SRB_int} and thus remove an interval of length $bn^{-1}\log n$ at each singularity $\hat\tau_i$, but this does not affect the error term. The above changes do not affect the rest of Section~\ref{sec:fluctuations}.

\medskip
The generalization to vector-valued observables is now achieved exactly as in the proof of Theorem~\ref{thm:gen1}.
\qed

%%%%%%%%%%%%%%%%%%%%%%%%%%%%%%%%%%%%
%%%%%%%%%%%%%%    References    %%%%%%%%%%%%%
%%%%%%%%%%%%%%%%%%%%%%%%%%%%%%%%%%%%

%\vskip 1cm

%\newpage

\bigskip
\bibliography{Quasistatic}{}
\bibliographystyle{plainurl}

%%%%%%%%%%%%%%%%%%%%%%%%%%%%%%%%%%%%

\vspace*{\fill}

\end{document}